\documentclass[smallcondensed,numbook,envcountsame,envcountreset]{svjour3} 

\journalname{Geometriae Dedicata}
\smartqed

\title{Dihedral manifold approximate fibrations\\ over the circle}
\titlerunning{Dihedral manifold approximate fibrations over the circle}

\author{Bruce Hughes \and Qayum Khan}
\authorrunning{B.~Hughes \and Q.~Khan}
\institute{B.~Hughes \and Q.~Khan\\ Department of Mathematics, Vanderbilt University, Nashville, TN 37240 U.S.A.
\\ \email{bruce.hughes@vanderbilt.edu \and qayum.khan@vanderbilt.edu}}

\dedication{Dedicated to Bruce Williams on the occasion of his 60th birthday}

\usepackage{amsmath,amssymb}
\usepackage{amscd,pb-diagram}
\usepackage{hyperref}
\usepackage{mathptmx}

\spnewtheorem*{untheorem}{Theorem}{\bf}{\it}
\spnewtheorem*{notation}{Notation}{\bf}{\rm}
\spnewtheorem{hypothesis}[theorem]{Hypothesis}{\bf}{\rm}

\newcommand{\br}{\ensuremath{\mathbb{R}}} 
\newcommand{\bz}{\ensuremath{\mathbb{Z}}} 
\newcommand{\bn}{\ensuremath{\mathbb{N}}} 
\newcommand{\id}{\ensuremath{\mathrm{id}}} 
\newcommand{\diam}{\ensuremath{\mathrm{diam}}} 
\newcommand{\co}{\ensuremath{\colon}} 
\newcommand{\incl}{\ensuremath{\mathrm{incl}}} 
\newcommand{\Shapes}{\ensuremath{\mathrm{Shapes}}} 

\newcommand{\cone}{\ensuremath{{\mathrm{c}}}} 
\newcommand{\oc}{\ensuremath{{\overset{\circ}{\mathrm{c}}}}} 
\newcommand{\ob}{\ensuremath{{\overset{\circ}{\mathrm{B}}}}} 
\newcommand{\proj}{\ensuremath{\mathrm{proj}}} 
\newcommand{\Graph}{\ensuremath{\mathrm{Graph}}} 
\newcommand{\ol}[1]{\overline{#1}} 
\newcommand{\norm}[1]{\| #1 \|}
\newcommand{\intr}{\mathrm{int}}

\newcommand{\cI}{\mathcal{I}}
\newcommand{\cJ}{\mathcal{J}}
\newcommand{\cP}{\mathcal{P}}
\newcommand{\cS}{\mathcal{S}}
\newcommand{\cT}{\mathcal{T}}
\newcommand{\cU}{\mathcal{U}}

\newcommand{\gens}[2]{\langle #1, #2 \rangle}

\begin{document}

\maketitle

\begin{abstract}
Consider the cyclic group $C_2$ of order two acting by
complex-conjugation on the unit circle $S^1$. The main result is
that a finitely dominated manifold $W$ of dimension  $>4$ admits a
cocompact, free, discontinuous action by the infinite dihedral group
$D_\infty$ if and only if $W$ is the infinite cyclic cover of a free
$C_2$-manifold $M$ such that $M$ admits a $C_2$-equivariant manifold
approximate fibration to $S^1$. The novelty in this setting is the
existence of codimension-one, invariant submanifolds of $M$ and $W$.
Along the way, we develop an equivariant sucking principle for
orthogonal actions of finite groups on Euclidean space.
\end{abstract}

\setcounter{tocdepth}{1} \tableofcontents

\section{Introduction}

We begin with a classical theorem, attributed to T.A.~Chapman
\cite{Chapman}, on sucking and wrapping up manifolds over the real
line.  Recall that the infinite cyclic group $C_\infty$ acts on the
real line $\br$ by integer translations.

\begin{untheorem}[Chapman]\label{thm:Chapman_equivalences}
Let $W$ be a connected topological manifold of dimension $>4$. The
following statements are equivalent:
\begin{enumerate}
\item
The space $W$ is finitely dominated, and there exists a
cocompact, free, discontinuous $C_\infty$-action on $W$.

\item
There exists a proper bounded fibration $W \to \br$.

\item
There exists a manifold approximate fibration $W \to \br$.

\item
There exist a $C_\infty$-action on $W$ and $C_\infty$-manifold
approximate fibration $W \to \br$.

\item
There exists a manifold approximate fibration $M \to S^1$ with
$\ol{M}$ homeomorphic to $W$.
\end{enumerate}
\end{untheorem}

This theorem can be viewed as an answer to the question:
\begin{quote}
\emph{When does a finitely dominated manifold $W$ admit a cocompact,
free, discontinuous action of the infinite cyclic group $C_\infty$?}
\end{quote}
There are two essentially different answers. The first, spelled out
by conditions (2) and (3), is that $W$  admits a proper map to $\br$
with bounded or controlled  versions of the homotopy lifting
property (called a bounded or approximate fibration). The
equivalence of the bounded and controlled versions is the main
advance of Chapman's paper \cite{Chapman} (which was preceded by the
Hilbert cube manifold case in \cite{ChapmanTAMS}), and it is part of
the phenomenon called \emph{sucking}.

The second answer, formulated in conditions (4) and (5), is that the
approximate fibration $W\to\br$ can be made equivariant with respect
to some $C_\infty$-action on $W$, or that $W\to\br$ can be
\emph{wrapped-up} to an approximate fibration $M\to S^1$. Chapman's
wrapping-up construction is a variation of Siebenmann's twist-gluing
construction (where the twist is the identity) used by him
\cite{Siebenmann} in his formulation of Farrell's fibering theorem
\cite{Farrell_thesis} (for more details see \cite{HughesRanicki}).
From this point of view, the question that is being raised is:
\begin{quote} \emph{Given a discrete subgroup $\Gamma$ of isometries on $\br$ and a manifold approximate fibration $W\to\br$, can the $\Gamma$-action on $\br$ be ``approximately lifted'' to a free, discontinuous $\Gamma$-action on $W$, so that there is a
$\Gamma$-manifold approximate fibration $W\to\br$?}
\end{quote}
Chapman proves that this can always be done when $\Gamma=C_\infty$
and $\dim W > 4$.

Since our formulation of Chapman's theorem does not appear
explicitly in \cite{Chapman}, we include a proof, which is actually
a guide to finding the proof in the literature. Most aspects of this
theorem were discovered independently by Ferry \cite{Ferry_notes}.
An analysis of many of the details appears in  \cite{HughesRanicki}.

\begin{proof}
The implication $(1) \Longrightarrow (2)$ follows from
Proposition~\ref{prop:existGmap} and then \cite[Proposition
17.14]{HughesRanicki}. The implication $(2) \Longrightarrow (3)$ is
the case $n=1$ in Corollary~\ref{cor:euclidean_delta_sucking}, which
is the $\epsilon$-$\delta$ version of Chapman's sucking principle.
The implication $(3) \Longrightarrow (2)$ is clear from the
definitions. The implication $(3) \Longrightarrow (4)$ is
\cite[Lemma 17.8]{HughesRanicki}, which uses Hughes's approximate
isotopy covering principle instead of an argument of Chapman
(cf.~\cite[Rmk.~17.9]{HughesRanicki}). The implication $(4)
\Longrightarrow (5)$ follows by taking the quotient of $W \to \br$
by the free $C_\infty$-action. The implication $(5) \Longrightarrow
(3)$ follows from taking the infinite cyclic cover of $M \to S^1$.
Finally, since $W$ has the proper homotopy type of a locally finite
simplicial complex \cite[Essay III, Thm.~4.1.3]{KirbySiebenmann},
the implication $(4) \Longrightarrow (1)$ follows from the
elementary argument of \cite[Prop.~17.12]{HughesRanicki}.\qed
\end{proof}

The theme of this paper is to extend Chapman's results from the
smallest infinite discrete group of isometries on $\br$, namely the
infinite cyclic group $C_\infty$, to the largest discrete group of
isometries on $\br$, namely the infinite dihedral group $D_\infty$.
In particular, our main result can be viewed as an answer to the
question:
\begin{quote}
\emph{When does a finitely dominated manifold $W$ admit a cocompact,
free, discontinuous action of the infinite dihedral group
$D_\infty$?}
\end{quote}

The main technical issues that arise involve the fact that
$D_\infty$ has torsion; it has a non-trivial finite subgroup, namely
the cyclic group $C_2$ of order two, which acts by reflection on
$\br$ fixing the origin. The lifted $D_\infty$-action on $W$
contains a $C_2$-action on $W$ with a nonempty invariant subset. For notation, for any $N \geq 1$, consider the dihedral groups
\begin{eqnarray*}
D_\infty &:=& C_\infty \rtimes_{-1} C_2\\
D_N &:=& C_N \rtimes_{-1} C_2.
\end{eqnarray*}
In particular, $D_1 = C_2$. There are short exact
sequences
\begin{gather*}
0 \longrightarrow C_\infty \longrightarrow D_\infty \longrightarrow
C_2 \longrightarrow 0\\
0 \longrightarrow N \cdot C_\infty \longrightarrow D_\infty
\longrightarrow D_N \longrightarrow 0.
\end{gather*}
 Fix presentations:
\begin{eqnarray*}
C_\infty &=& \langle T ~|~ \rangle \\
C_2 &=& \langle R ~|~ R^2= 1 \rangle \\
D_\infty &=& \langle R, T ~|~ R^2= 1, RT = T^{-1}R \rangle.
\end{eqnarray*}
The $D_\infty$-action on $\br$ by isometries is given by $R(x)=-x$
and $T(x)=x+1$.

Our main theorem, which we now state, contains analogues of all
parts of Chapman's theorem.

\begin{theorem}[Main Theorem]
\label{intro thm:main} Let $W$ be a connected topological manifold
of dimension $>4$. The following statements are equivalent:
\begin{enumerate}
\item
The space $W$ is finitely dominated, and there exists a
cocompact, free, discontinuous $D_\infty$-action on $W$.

\item
There exist a free $C_2$-action on $W$ and proper $C_2$-bounded
fibration $W \to \br$.

\item
There exist a free $C_2$-action on $W$ and $C_2$-manifold
approximate fibration $W \to \br$.

\item
There exist a free $D_\infty$-action on $W$ and
$D_\infty$-manifold approximate fibration $W\to \br$.

\item
For every $N \geq 1$, there exist a free $D_N$-action on a
manifold $M$ and $D_N$-manifold approximate fibration $M \to
S^1$ with induced infinite cyclic cover $\ol{M}$ 
homeomorphic to $W$.
\end{enumerate}
\end{theorem}

The prototypical example of a manifold $W$ satisfying the conditions
of Theorem~\ref{intro thm:main} is the product $W=S^n\times\br$,
where the $D_\infty$-action on $W$ is given by $T(x,t)=(x,t+1)$ and
$R(x,t) = (-x,-t)$.

Here is a guide to finding the proof of the theorem in the rest of
the paper.

\begin{proof}
The implication $(1) \Longrightarrow (2)$ follows from
Proposition~\ref{prop:existGmap} and then Theorem
\ref{thm:unwrapping}. The implication $(2) \Longrightarrow (3)$ is
the special case of $G=C_2$ and $n=1$ in Corollary
\ref{cor:equivariant_bounded_sucking}, which requires freeness. The
implication $(3) \Longrightarrow (2)$ is clear from the definitions.
The implication $(3) \Longrightarrow (4)$ is
Theorem~\ref{thm:WrappingUp}. The implication $(4) \Longrightarrow
(5)$ follows by taking the quotient by the action of the subgroup $N
\cdot C_\infty$, where $N \geq 1$. The implication $(5)
\Longrightarrow (3)$ follows from taking the infinite cyclic cover
of $M \to S^1$ for the special case $D_1 = C_2$.

Finally, consider the implication $(4) \Longrightarrow (1)$.
Restrict the $D_\infty$-action on $W$ and the $D_\infty$-manifold
approximate fibration $p: W \to \br$ to the finite-index subgroup
$C_\infty$. Therefore, by the implication $(4) \Longrightarrow (1)$
in Theorem \ref{thm:Chapman_equivalences}, we obtain that $W$ is
finitely dominated and the $C_\infty$-action (hence the
$D_\infty$-action) on $W$ is cocompact, free, and discontinuous.\qed
\end{proof}

Most of the work in this paper applies to settings more general than
the real line in the main theorem. We now discuss some of the
highlights.

The first is a variation on the well-known property that close maps
into an ANR $B$ are closely homotopic. In the compact case,
closeness is measured uniformly by using a metric on $B$. In the
non-compact case, open covers of $B$ are used (see
Section~\ref{subsec:Metrically close maps into ANRs}). In this paper
we are confronted with metrically close maps into a non-compact ANR
$(B,d)$ which we want to conclude are metrically closely homotopic.
We isolate a condition, called \emph{finite isometry type}, that
allows us to do this. Here is the result.

\begin{theorem}
Suppose $(B,d)$ is a triangulated metric space with finite isometry
type. For every $\epsilon >0$ there exists $\delta > 0$ such that:
if $X$ is any space and $f,g\co X\to B$ are two close $\delta$-close
maps, then $f$ and $g$ are $\epsilon$-homotopic rel
\[
X_{(f=g)} := \{ x\in X ~|~ f(x) =g(x)\}.
\]
\end{theorem}

See Section~\ref{subsec:Finite_isometry_type} for the definition of
finite isometry type and Section~\ref{subsec:Metrically close maps
into ANRs} for a proof of the theorem (Proposition~\ref{prop:metric
ANR}).

Likewise, we need a metric version of Chapman's manifold approximate
fibration sucking theorem, which has an open cover formulation in
Chapman's work. We again find the finite isometry type condition
suitable for our purposes.

\begin{theorem}[Metric MAF Sucking]
\label{intro thm:metric maf sucking} Suppose $(B,d)$ is a
triangulated metric space with finite isometry type, and let $m >
4$. For every $\epsilon >0$ there exists $\delta > 0$ such that: if
$M$ is an $m$-dimensional manifold and $p\co M\to B$ is a proper
$\delta$-fibration, then $p$ is $\epsilon$-homotopic to a manifold
approximate fibration $p'\co M\to B$.
\end{theorem}

The proof of Theorem~\ref{intro thm:metric maf sucking} is found in
Section~\ref{subsec:Metric version of Chapman's MAF sucking theorem}
(use Corollary ~\ref{cor:FIT-nonrel},
Proposition~\ref{prop:metric ANR}).

The following result is an equivariant version of Theorem~\ref{intro
thm:metric maf sucking}, valid for many finite subgroups of $O(n)$.
We will need it only for the simplest non-trivial case of $C_2\leq
O(1)$ when we perform dihedral wrapping up. However, we expect
future applications in the theory of locally linear actions of
finite groups on manifolds.

\begin{theorem}[Orthogonal Sucking]
\label{intro thm: ortho sucking} Suppose $G$ is a finite subgroup of
$O(n)$ acting freely on $S^{n-1}$, and let $m > 4$. For every $\epsilon > 0$ there
exists $\delta > 0$ such that: if $M$ is a free $G$-manifold of
dimension $m$ and $p\co M\to\br^n$ is a proper
$G$-$\delta$-fibration, then $p$ is $G$-$\epsilon$-homotopic to a
$G$-manifold approximate fibration $p': M \to \br^n$.
\end{theorem}

The proof of Theorem~\ref{intro thm: ortho sucking} is located in
Section~\ref{section:sucking}.

Let us consider an example that illustrates how one may enter into
the situation of Theorem~\ref{intro thm:main}. Consider an
equivariant version of Farrell's celebrated fibering theorem
\cite{Farrell_thesis} \cite{Farrell_ICM}, which solves the problem of when
a map $f: M\to S^1$ from a high-dimensional manifold $M$ to the
circle $S^1$ is homotopic to a fiber bundle projection. One usually
inputs the necessary condition that $\overline{M}$ be finitely
dominated. The first non-trivial equivariant version of this problem
concerns the action of $C_2$ on $S^1$ by complex conjugation and a
free action of $C_2$ on $M$. The prototypical example of such a
$C_2$-fiber bundle is the projection $p: S^n\times S^1\to S^1$,
where the free $C_2$-action on $S^n\times S^1$ is given by
$R(x,z)=(-x,\bar z)$. Observe that the orbit space is a connected
sum of projective spaces, $(S^n\times S^1)/C_2 = \br P^{n+1}\#\br
P^{n+1}$. With this action of $C_2$ on $M=S^n\times S^1$, the
infinite cyclic cover is $\overline{p}:
\overline{M}=S^n\times\br\to\br$ with the $D_\infty$-action
described above. In the language of
Section~\ref{sec:bounded_fibrations}, $(M,f,R)$ is a $C_2$-manifold
band. It is shown in Theorem~\ref{thm:unwrapping} that
$\overline{f}: \overline{M} \to \br$ is a $C_2$-bounded fibration.
Thus, Theorem~\ref{intro thm:main} is a preliminary step in
constructing a $C_2$-relaxation of a $C_2$-manifold band, following
Siebenmann's \cite{Siebenmann} approach in the non-equivariant case.
We expect to return to the $C_2$-fibering problem in a future paper.

\section{Preliminary notions}

This section contains most of the standard terminology that we will
use throughout the paper. In addition, there are statements of two
known principles that we will need: relative sucking and approximate
isotopy covering.

By \emph{manifold} we mean a metrizable topological manifold without
boundary. An action of a discrete group $\Gamma$ on a topological
space $X$ is \emph{discontinuous} if, for all $x \in X$, there
exists a neighborhood $U$ of $x$ in $X$ such that the set $\{g \in
\Gamma ~|~ U \cap g(U) = \varnothing \}$ is finite. The action is
\emph{cocompact} if $X/\Gamma$ is compact. In this paper, every
\emph{ANR} (absolute neighborhood retract) is understood to be
locally compact, separable, and metrizable.

We denote the \emph{natural numbers} by $\bn := \bz_{\geq 0} = \{
0,1,2,3,\ldots\}$, and the \emph{unit interval} by $I := [0,1]$. For
each $k\in\bn$ and $r>0$, denote the \emph{open} and \emph{closed
cubes}:
\[
\ob{}_r^k := (-r,r)^k \subseteq \mathrm{B}_r^k :=
[-r,r]^k\subseteq\br^k.
\]
We denote the various \emph{open cones} on a topological space $X$
as follows:
\begin{eqnarray*}
\oc(X) &:=& (X\times [0,\infty))/ (X \times \{0\}) \\
\cone_r(X) &:=& (X\times [0,r])/ (X \times \{0\}) \\
\oc_r(X) &:=& (X\times [0,r))/ (X \times \{0\}) ~.
\end{eqnarray*}

If $(Y,d)$ is a metric space, $\epsilon > 0$, and $H\co X\times I\to
Y$ is a homotopy, then $H$ is an \emph{$\epsilon$-homotopy} if the
diameter of $H(\{x\}\times I)$ is less than $\epsilon$ for all $x\in
X$. The homotopy $H$ is a \emph{bounded homotopy} if it is an
$\epsilon$-homotopy for some $\epsilon >0$. Given two maps $u,v: W
\to Y$, we say that \emph{$u$ is $\epsilon$-close to $v$} if
$d(u(w),v(w)) < \epsilon$ for all $w \in W$. Furthermore, given
$\mu>0$ and $A \subseteq Y$, we say that \emph{$u$ is
$(\epsilon,\mu)$-close to $v$ with respect to $A$} if $u$ is
$\epsilon$-close to $v$ and $u|v^{-1}(A)$ is $\mu$-close to
$v|v^{-1}(A)$.

\subsection{Lifting problems}
The notion of approximate fibration is due to Coram--Duvall
\cite{CoramDuvall}. A convenient source for the following
definitions is \cite[\S16]{HughesRanicki}.

\begin{definition} Let $(Y,d)$ be a metric space, and let $p\co X\to Y$ be a map.
\begin{enumerate}
\item Let $\epsilon > 0$. The map
$p$ is an \textbf{$\epsilon$-fibration} if, for every space $Z$
and commutative diagram
$$
\begin{diagram}
\node{Z} \arrow{s,tb}{\times 0}{} \arrow{e,t}{f} \node{X} \arrow{s,b}{p} \\
\node{Z\times I}\arrow{ne,t,..}{\tilde F} \arrow{e,t}{F} \node{Y}
\end{diagram}
$$
there exists an \textbf{$\epsilon$-solution}. That is, there
exists a homotopy $\tilde F : Z\times I\to X$ such that $\tilde
F(-,0) = f$ and $p\tilde F$ is $\epsilon$-close to $F$. The
above diagram is a \textbf{homotopy lifting problem}.

\item Let $\epsilon >0$ and $A \subseteq Y$. The map $p$ is an \textbf{$\epsilon$-fibration over $A$}
if every homotopy lifting problem with $F(Z \times I) \subseteq
A$ has an \textbf{$\epsilon$-solution}.  In particular, the map
$p$ is a \textbf{$\epsilon$-fibration} if $p$ is a
$\epsilon$-fibration over $Y$.

\item The map $p$ is a
\textbf{bounded fibration} if $p$ is an $\epsilon$-fibration for
some $\epsilon >0$. The map $p$ is a \textbf{manifold
approximate fibration (MAF)} if $X$ and $Y$ are topological
manifolds and $p$ is a proper $\epsilon$-fibration for every
$\epsilon>0$.
\end{enumerate}
\end{definition}

\begin{remark}
By an {\it $\epsilon$-fibration}, Chapman \cite{Chapman} means a
proper map between ANRs, that is, an $\epsilon$-fibration for the
class of locally compact, separable, metric spaces. This restriction
on the class of spaces is of no consequence in the present paper as
can be seen in either of two ways. First, the reader will note that
the only property we use of $\epsilon$-fibrations is the
$\epsilon$-homotopy lifting property for compact, metric spaces.
Second, the techniques of Coram--Duvall \cite{CoramDuvall2} can be
adapted to show: for a proper map $E\to B$ between ANRs, the
$\delta$-homotopy lifting property for compact, metric spaces
implies the $\epsilon$-homotopy lifting property for all spaces
(where $\delta >0$ depends only on $\epsilon
>0$ and the metric on $B$).
\end{remark}

\subsection{Equivariant lifting problems}
The notion of $(\epsilon,\nu)$-fibration is due to Hughes
\cite{Hughes1985Q}.  See Prassidis \cite{Prassidis} for a more
complete treatment.

Let $G$ be a group. A \emph{$G$-space} is a topological space with a
left $G$-action, and a \emph{$G$-map} between $G$-spaces is a
continuous, equivariant map. Throughout the paper, we assume that
the unit interval $I$ is a trivial $G$-space.

\begin{definition}
Let $(Y,d)$ be a metric space, and let $p\co X\to Y$ be a $G$-map.
\begin{enumerate}
\item
Let $\epsilon > 0$. The map $p$ is a
\textbf{$G$-$\epsilon$-fibration} if, for every $G$-space $Z$
and commutative diagram of $G$-maps:
$$
\begin{diagram}
\node{Z} \arrow{s,tb}{\times 0}{} \arrow{e,t}{f} \node{X} \arrow{s,b}{p} \\
\node{Z\times I}\arrow{ne,t,..}{\tilde F} \arrow{e,t}{F} \node{Y}
\end{diagram}
$$
there exists a \textbf{$G$-$\epsilon$-solution}. That is, there
exists a $G$-homotopy $\tilde F : Z\times I\to X$ such that
$\tilde F (-,0) = f$ and $p\tilde F$ is $\epsilon$-close to $F$.
The diagram above is called a \textbf{$G$-homotopy lifting
problem}.

\item Let $\epsilon >0$ and $A \subseteq Y$. The map $p$ is a \textbf{$G$-$\epsilon$-fibration over $A$}
if every $G$-homotopy lifting problem with $F(Z \times I)
\subseteq A$ has a \textbf{$G$-$\epsilon$-solution}. In
particular, the map $p$ is a \textbf{$G$-$\epsilon$-fibration}
if $p$ is a $G$-$\epsilon$-fibration over $Y$.

\item The map $p$ is a \textbf{$G$-bounded fibration} if $p$ is a $G$-$\epsilon$-fibration
for some $\epsilon > 0$. The map $p$ is a \textbf{$G$-manifold
approximate fibration ($G$-MAF)} if $X$ and $Y$ are topological
manifolds and $p$ is a proper $G$-$\epsilon$-fibration for every
$\epsilon > 0$.

\item Let $\epsilon, \mu >0$ and $A \subseteq B \subseteq Y$.  The map $p$ is a \textbf{$G$-$(\epsilon, \mu)$-fibration
over $(B,A)$} if every $G$-homotopy lifting problem with $F(Z
\times I) \subseteq B$ has a
\textbf{$G$-$(\epsilon,\mu)$-solution over $(B,A)$}. That is,
there exists a $G$-homotopy $\tilde{F} : Z\times I\to X$ such
that $\tilde{F}(-,0) = f$ and $p\tilde{F}$ is
$(\epsilon,\mu)$-close to $F$ with respect to $A$.
\end{enumerate}
\end{definition}

\subsection{Stratified lifting problems}
See Hughes \cite{Hughes2002} for the definitions here.

\begin{definition}
Let $X$ be a topological space.
\begin{enumerate}
\item
A \textbf{stratification of $X$} is a locally finite partition
$\{ X_i\}_{i\in \cI}$ for some set $\cI$ such that each $X_i$,
called the \textbf{$i$-stratum}, is a locally closed subspace of
$X$. We call $X$ a \textbf{stratified space}.
\item
If $X$ is a stratified space, then a map $f:Z\times A\to X$ is
\textbf{stratum-preserving along $A$} if, for each $z\in Z$, the
image $f(\{ z\}\times A)$ lies in a single stratum of $X$. In
particular, a map $f:Z\times I\to X$ is a
\textbf{stratum-preserving homotopy} if $f$ is
stratum-preserving along $I$.
\end{enumerate}
\end{definition}

\begin{definition}
Let $(Y,d)$ be a metric space, and let $p:X\to Y$ be a map. Suppose
$X$ and $Y$ are spaces with stratifications $\{ X_i\}_{i\in \cI}$
and $\{ Y_j\}_{j\in \cJ}$.
\begin{enumerate}
\item
The map $p$ is a \textbf{stratified fibration} if, for every
commutative diagram
$$
\begin{diagram}
\node{Z} \arrow{s,tb}{\times 0}{} \arrow{e,t}{f} \node{X} \arrow{s,b}{p} \\
\node{Z\times I}\arrow{ne,t,..}{\tilde F} \arrow{e,t}{F} \node{Y}
\end{diagram}
$$
with $F$ a stratum-preserving homotopy, there exists a
\textbf{stratified solution}. That is, there exists a
stratum-preserving homotopy $\tilde F : Z\times I\to X$ such
that $\tilde F (-,0) = f$ and $p\tilde F= F$. The diagram above
is called a \textbf{stratified homotopy lifting problem}.

\item Let $\epsilon >0$. The map $p$ is a \textbf{stratified
$\epsilon$-fibration} if, for every stratified homotopy lifting
problem, there exists a \textbf{stratified $\epsilon$-solution}.
That is, there exists a stratum-preserving homotopy $\tilde F :
Z\times I\to X$ such that $\tilde F(-,0) = f$ and $p\tilde F$ is
$\epsilon$-close to $F$.
\end{enumerate}
\end{definition}

\subsection{Relative sucking}

In Section~\ref{sec:wrappingup} we will need the following relative
version of Chapman's MAF Sucking Theorem \cite[Theorem 1]{Chapman}.
We include a proof that shows it is a formal consequence of
Chapman's work.

\begin{proposition}\label{prop:rel_sucking}
Suppose $a_1 < a_3 <b_3 <b_1$  are real numbers. For
every $m > 4$ there exists $\epsilon > 0$ such that: if $W$
is an $m$-dimensional manifold, $p\co W\to\br$ is a proper
$\epsilon$-fibration, and $p$ is an approximate fibration over
$(-\infty, a_3)\cup (b_3,\infty)$, then there exist a manifold
approximate fibration $p'\co W\to \br$ and a homotopy $p\simeq p'$
rel $p^{-1}((-\infty,a_1]\cup [b_1,\infty))$.
\end{proposition}

\begin{proof}
Let $m > 4$ be given. Let $\mathcal U$ be an open cover of
$(a_1,b_1)$ such that: if $p\co X\to \br$ is a map and
$p|p^{-1}(a_1,b_1)$ is $\mathcal U$-homotopic to a map $q:
p^{-1}(a_1,b_1)\to(a_1,b_1)$ via a homotopy $H$, then $H$ extends to
a homotopy $\widetilde{H}\co p\simeq\widetilde{q}$ rel
$p^{-1}((-\infty,a_1]\cup [b_1,\infty))$. For example, this property
is satisfied by the open cover
\[
\mathcal U := \{(x-r_x,x+r_x) | x \in (a_1,b_1) \} \text{ where } r_x :=
{\textstyle \frac{1}{2}}\min\{x-a_1, b_1-x\}.
\]
By Chapman's MAF Sucking Theorem \cite[Thm.~1]{Chapman} (see
Theorem~\ref{Chapman's MAF Sucking Theorem} below), there exists an
open cover $\mathcal V$ of $(a_1,b_1)$ such that: if $M$ is an
$m$-dimensional manifold and $p\co M\to (a_1,b_1)$ is a proper
$\mathcal V$-fibration, then $p$ is $\mathcal U$-homotopic to a MAF.

Next, select $a_2 \in (a_1, a_3)$ and $b_2 \in (b_3, b_1)$. Consider an open cover $\mathcal{O}$ of $(a_1,b_1)$:
\[
\mathcal O :=\left\{ (a_1, (a_2+a_3)/2), (a_2,b_2), ((b_3+b_2)/2,
b_1)\right\}.
\]
By a local-global result of Chapman \cite[Prop.~2.2]{ChapmanTAMS}
\cite[Prop.~2.2]{Chapman}, there exists an open cover $\mathcal W$
of $(a_1,b_1)$ such that any proper map of an ANR to $(a_1,b_1)$
restricting to a $\mathcal W$-fibration over the closure of each
member of $\mathcal O$ is a $\mathcal V$-fibration. Finally, choose
$\epsilon
>0$ such that the $\epsilon$-ball about any point of $[a_2,b_2]$ is
contained in some member of $\mathcal W$. One checks that $\epsilon$
satisfies the desired conditions.\qed
\end{proof}

\subsection{Approximate isotopy covering}

We begin with the 1-simplex version of the Approximation Theorem
\cite[Thm.~14.1]{HTW1990}.

\begin{theorem}[Hughes]\label{thm:Approximation}
Let $p: M \times I \to B \times I$ be a fiber-preserving MAF such
that $\dim\,M>4$. For every open cover $\delta$ of $B$, there exists
a fiber-preserving homeomorphism $H: M \times I \to M \times I$ such
that $H_0 = \id_M$ and $pH$ is $\delta$-close to $p_0 \times \id_I$.
\qed
\end{theorem}

In \cite[Theorem~7.1]{Hughes1985T}, the proof of the above theorem
is given only for $\epsilon>0$ and closed manifolds $B$ that admit a
handlebody decomposition (e.g. a compact, piecewise linear manifold
without boundary). However, as observed in \cite[\S14]{HTW1990},
that proof (by induction on the index of handles) adapts to give a
proof for any connected manifold $B$ that admits a handlebody
decomposition. Moreover, the arguments in \cite[\S13]{HTW1990} can
be adapted to prove the general case of $B$ any topological
manifold. In this paper, we only require the case of $B=\br$.

A well-known consequence of Hughes's Approximation Theorem is the
Approximate Isotopy Covering Theorem for MAFs (see
\cite[\S6]{HTW1993} \cite[Thm.~17.4]{HughesRanicki}). The following
metric version is an immediate consequence of
Theorem~\ref{thm:Approximation}.

\begin{corollary}[Approximate Isotopy Covering]\label{cor:AIC}
Let $p: M \to B$ be a MAF with $\dim\,M>4$. Suppose $g: B
\times I \to B$ is a isotopy from the identity $\id_B$. For every
$\epsilon > 0$, there exists an isotopy $G: M \times I \to M$ from
the identity $\id_M$ such that $pG_s$ is $\epsilon$-close to $g_s p$
for all $s \in I$.
\end{corollary}

\begin{proof}
Define $g'\co B\times I\to B\times I$ by $g'(x,s) = (g_s(x),s)$.
Then $g'(p\times\id_I)\co M\times I\to B\times I$ is a
fiber-preserving MAF. Given $\epsilon >0$, by
Theorem~\ref{thm:Approximation}, there exists a fiber-preserving
homeomorphism $H\co M\times I\to M\times I$ such that $H_0=\id_M$
and $g'(p\times\id_I)H$ is $\epsilon$-close to
$g'(p\times\id_I)_0\times\id_I= p\times\id_I$. It follows that $G\co
M\times I\to M$ defined by $G(x,s) = H_s^{-1}(x)$ is the desired
isotopy.\qed
\end{proof}

\section{Finite isometry type and metric sucking}

In this section we introduce a condition on a triangulated metric
space, \emph{finite isometry type}, which essentially says that the
local geometry of the space has finite variation. This special
condition allows us to phrase two controlled-topological results
about non-compact triangulated spaces in terms of the metric rather
than open covers. The two results concern close maps into an ANR and
Chapman's MAF Sucking Theorem.

\subsection{Finite isometry type}\label{subsec:Finite_isometry_type}

We begin by introducing terminology for the metric conditions to
appear in the results of Sections~\ref{subsec:Metrically close maps
into ANRs} and \ref{subsec:Metric version of Chapman's MAF sucking
theorem}. This will allow for cleaner statements of our results on
equivariant sucking in Section~\ref{section:sucking}.

The following notion of shapes is used by Bridson--Haefliger
\cite[Chapter I.7]{BridsonHaefliger} in a slightly more restrictive
context.

\begin{definition}
A \textbf{simplicial isometry} between two metrized simplices is an
isometry that takes each face onto a face. If $(B,d)$ is a
triangulated metric space, the {\bf shapes of $B$}, denoted
$\Shapes^\triangle(B)$, is the set of simplicial isometry classes of
simplices of $B$.
\end{definition}

\begin{definition}
If $v$ is a vertex of a simplicial complex $K$, then $K_v$ denotes
the \textbf{closed star of $v$ in $K$}. That is, $K_v$ is the union of
all simplices of $K$ that contain $v$, triangulated by the simplicial complex
consisting of all simplices of $K$ that are faces of simplices having $v$ as
a vertex.
\end{definition}

We often abuse notation by identifying a simplicial complex with its underlying
polyhedron.

\begin{definition}
Suppose $(B,d)$ is a metric space with a triangulation $\phi\co K\to
B$ (i.e., $K$ is a simplicial complex and $\phi$ is a
homeomorphism). For vertices $v, w\in K$, the closed stars
$\phi(K_v)$ and $\phi(K_w)$ in $B$ are {\bf  simplicially isometric}
if there exists an isometry $h\co \phi(K_v)\to \phi(K_w)$ such that
$\phi^{-1}h\phi\co K_v\to K_w$ is a simplicial isomorphism. The {\bf
shapes of closed stars of $B$}, denoted $\Shapes^\boxtimes(B)$, is
the set of simplicial isometry classes of closed stars of $B$. More
generally, if $A$ is a closed sub-polyhedron of $B$, then
$\Shapes^\boxtimes(B, A)$ denotes the set of simplicial isometry
classes of closed stars $\phi(K_v)$ of $B$ such that $\phi(v)\in A$.
\end{definition}

\begin{remark}\label{rem:finite}
If  $\phi\co K\to B$ is a triangulation of $(B,d)$ with respect to
which $\Shapes^\boxtimes(B)$ is finite, then $\Shapes^\triangle(B)$
is also finite. Moreover, if $K'$ is the first barycentric
subdivision of $K$, then $\Shapes^\boxtimes(B)$ is finite with
respect to the induced triangulation $\phi\co K'\to B$.
\end{remark}

\begin{definition}
\label{def:FIT} A complete metric space $(B,d)$ with a triangulation
$\phi\co K\to B$ has \textbf{finite isometry type} if
\begin{enumerate}
    \item $K$ is a locally finite, simplicial complex,
    \item there exists $d_0 >0$ such that: if $v \neq w$ are vertices of $B$ then $d(v,w)\geq d_0$,
    \item for every $\alpha > 0$ and for every $n\geq 0$, there exists $\beta >0$ such that: if $x,y$ lie in distinct
    $n$-simplices of $B$ and $d(x,y)<\beta$, then $x$ and $y$
    are in the $\alpha$-neighborhood of the $(n-1)$-skeleton of
    $B$, and
    \item $\Shapes^\boxtimes(B)$ is finite.
\end{enumerate}
\end{definition}

\begin{remark}\label{rem:FIT1}
A finite triangulation has finite isometry type.
\end{remark}

\begin{remark}\label{rem:FIT2}
Let $(R,e)$ be a metric space, and let $p: R\to M$ be a regular cover with $M$ compact. Suppose the group $D$ of covering transformations acts by isometries on $(R,e)$. Observe that $p: (R,e) \to (M,e/D)$ is distance non-increasing and a local isometry. Also, $p$ is regular implies that $D$ acts transitively on the fibers.
Therefore, if $\cS$ is a finite triangulation of $M$, then the induced $D$-equivariant triangulation $p^*(\cS)$ of $R$ has finite isometry type.
\end{remark}

This finiteness condition arises in the following rigid situations,  not pursued here.

\begin{example}
\label{ex:FIT} Let $(B,d)$ be a triangulated, geodesic metric space
such that each simplex of $B$ is a geodesic subspace. It follows
that the metric on each closed star in $B$ is completely determined
by the metric on the simplices in the star. Thus, if
$\Shapes^\triangle(B)$ is finite, then so is $\Shapes^\boxtimes(B)$.
It is also clear that for such $B$, the condition
$\Shapes^\triangle(B)$ is finite implies conditions (2) and (3) of
Definition~\ref{def:FIT}. In summary, if $(B,d)$ is a complete,
geodesic metric space triangulated by a locally finite simplicial
complex, each simplex is a geodesic subspace, and
$\Shapes^\triangle(B)$ is finite, then $(B,d)$ has finite isometry type.
\end{example}

\begin{example}
Let $K$ be a connected $M_\kappa$-simplicial complex in the sense of
Bridson--Haefliger \cite{BridsonHaefliger}. If
$\Shapes^\triangle(K)$ is finite, then it follows from \cite[Theorem
7.19, page 105]{BridsonHaefliger} that $(K,d)$ is a complete
geodesic space, where $d$ is the intrinsic metric. Thus, by applying
Example~\ref{ex:FIT}, it follows that $(K,d)$ has finite isometry
type.
\end{example}

We show that
split crystallographic groups induce triangulations of Euclidean space
$(\br^n, e)$ with the standard metric with respect to which $(\br^n,
e)$ has finite isometry type. 

\begin{definition}
An action of a group $\Gamma$ on a topological space $R$ is
\textbf{virtually free} if there is a finite-index subgroup $\Delta
\leq \Gamma$ whose restricted action on $R$ is free.
\end{definition}

\begin{proposition}\label{prop:triangulation}
Let $(R,e)$ be a metric space such that $R$ is a smooth manifold.
Let $\Gamma$ be a discrete, cocompact subgroup of smooth isometries
of $(R,e)$. If the action of $\Gamma$ on $R$ is virtually free, then
$R$ admits a smooth, $\Gamma$-equivariant triangulation $\cT$ of
finite isometry type, which is unique up to $\Gamma$-combinatorial equivalence.
\end{proposition}

\begin{proof}
Since $\Gamma$ has a virtually free action on $R$, there exists a
finite-index subgroup $\Delta$ of $\Gamma$ with a free action on
$R$. Define a finite-index, normal subgroup $\Delta_0$ of $\Gamma$
by
\[
\Delta_0 := \bigcap_{g \in \Gamma} g \Delta g^{-1}.
\]
Then there is an exact sequence $1 \to \Delta_0 \to \Gamma \to G \to
1$, where $G := \Gamma/\Delta_0$ is a finite group. Since $\Delta_0$
has a cocompact, free action on $R$, the quotient $M := R/\Delta_0$
is a compact smooth $G$-manifold. So, by a theorem of S.~Illman
\cite{Illman_Finite}, there exists a smooth, $G$-equivariant
triangulation $\cS$ of $M$ which is unique up to $G$-combinatorial equivalence. Since $\Delta_0 \to R \to M$ is the sequence of a
regular covering map, there is a unique, smooth,
$\Gamma$-equivariant triangulation $\cT$ of $R$ covering $\cS$.
Therefore, by Remark~\ref{rem:FIT2}, since $\cS$ is finite, we conclude that the pullback $\cT$ has finite isometry
type.\qed
\end{proof}

\begin{example}
For any crystallographic group $\Gamma$ of rank $n$, the Euclidean
space $(\br^n,e)$, where $e$ is the standard metric, admits a
$\Gamma$-equivariant triangulation $\cT$ of finite isometry type.
Indeed, the action is virtually free, since there is an exact
sequence $1 \to \bz^n \to \Gamma \to G \to 1$, where the subgroup
$\bz^n$ acts freely on $\br^n$ by translations, and the finite group
$G$ is called the \emph{point group} (see \cite{Farkas}). Suppose
this short exact sequence splits, in which case we call $\Gamma$ a
\textbf{split} crystallographic group. Then $G$ is a subgroup of the
orthogonal group $O(n)$ and the triangulation $\cT/G$ of
$(\br^n/G,e/G)$ has finite isometry type.
\end{example}

More generally, we have an important geometric observation used in Section~\ref{section:sucking}.

\begin{proposition}\label{prop:HenriquesLeary}
Let $G$ be any finite subgroup of $O(n)$. Then $\br^n/G$ has finite isometry type.
\end{proposition}

\begin{proof}
Consider the unit sphere $S^{n-1} \subset \br^n$ with the spherical metric $s$ and induced isometric $G$-action.  For any point $a \in S^{n-1}$, recall its \emph{Dirichlet domain} $D(a)$ is the open neighborhood
\[
D(a) := \{ x \in S^{n-1} ~|~ \forall g \in G ~:~ a \neq ga \Longrightarrow d(x,a) < d(x,ga) \}.
\]
Then, by \cite[Theorem 6.7.1]{RatcliffeBook}, the closure $\overline{D(a)}$ is a convex polyhedron and a fundamental domain for the $G$-action on $(S^{n-1},s)$. Select a geodesic triangulation of $\overline{D(a)}$. This extends to a geodesic $G$-triangulation of $S^{n-1}$. That is, we obtain a $G$-homeomorphism $\phi: K \to S^{n-1}$, from a $G$-simplicial complex $K$, such that the image $\phi(\sigma)$ of each simplex $\sigma$ of $K$ is totally geodesic in $(S^{n-1},s)$.  Upon replacing $K$ with a (necessarily $G$-equivariant) barycentric subdivision, we may assume that $s(\phi v,\phi v') < 1$ for all vertices $v,v'$ that share a simplex $\sigma$ of $K$. Now define $\psi: K \to \br^n$ as the unique continuous $G$-map such that $\psi(v):=\phi(v)$ for each vertex $v$ of $K$ and that the image $\psi(\sigma)$ of each simplex $\sigma$ of $K$ is the convex hull in $(\br^n,e)$. It is clear that $0 \notin \psi(K)$ and $\psi(\sigma) \cap \psi(\sigma') \subseteq \psi(\partial\sigma) \cap \psi(\partial \sigma')$ for all simplices $\sigma,\sigma'$ of $K$. Hence $\psi$ is injective and extends to a $G$-homeomorphism
\[
\Psi := \oc(\psi): \oc(K) \longrightarrow \br^n ~;~ (x,t) \longmapsto t \psi(x).
\]

Let $\cT$ be (the set of simplices of) a linear triangulation of $\br^n$ such that $\cT$ is invariant under permutation of coordinates
of $\br^n$ and under the standard action of $\bz^n$ on $\br^n$.
For example, $\cT$ can be taken to be the standard triangulation defined as follows. The set of vertices of $\cT$ is $\bz^n$.
There is a directed edge from a vertex $x=(x_1, x_2, \dots x_n)$ to a vertex $y= (y_1, y_2, \dots, y_n)$ if and only if $x \neq y$ and
$x_i\leq y_i\leq x_i+1$ for each $i=1, 2, \dots, n$.  
A finite set $\sigma$ of vertices spans a simplex if and only if for any two vertices in $\sigma$, there is a directed edge from one to the other.
It follows from the invariance properties that $\cT$ is a triangulation of finite isometry type. 
Moreover, if  $A\co\br^n\to\br^n$ is any non-singular linear transformation, then $A(\cT) = \{ A(\sigma) ~|~ \sigma\in\cT\}$ is also
a triangulation of finite isometry type.

Now, for each $(n-1)$-simplex $\sigma$ of $K$, define a non-singular linear transformation
$A_\sigma\co\br^n\to\br^n$ as follows. Order the vertices $v_1, v_2,\dots, v_n$ of $\Psi(\sigma)$, and define
$A_\sigma(e_i) := v_i$ for each $i=1,2,\dots, n$, where $e_1, e_2, \dots, e_n$ are the standard basis vectors of $\br^n$.
The triangulation $A_\sigma(\cT)$ of $\br^n$ restricts to a triangulation $\cT_\sigma$ of 
$\Psi(\oc(\sigma))$ that is independent of the ordering of  $v_1, v_2,\dots, v_n$.
Moreover, if $\sigma$ and $\tau$ are two $(n-1)$-simplices of $K$, then $\cT_\sigma$ and $\cT_\tau$ agree on 
$\Psi(\oc(\sigma\cap\tau))$.
It follow that $\cT_\Psi := \bigcup_\sigma\cT_\sigma$ is a $G$-equivariant triangulation of $\br^n$ of finite isometry type 
and induces a triangulation  of $\br^n/G$ of finite isometry type.\qed
\end{proof}

\subsection{Metrically close maps into ANRs}\label{subsec:Metrically close
maps into ANRs}

For notation if $f,g\co X\to Y$ are maps, then $X_{(f=g)}=\{ x\in X
~|~ f(x) = g(x)\}$. Recall the following classic property of ANRs
(cf.~\cite[p.~39]{MardesicSegal}).

\begin{proposition}
\label{prop:classic ANR} Let $B$ be an ANR. For every open cover
$\mathcal U$ of $B$ there exists an open cover $\mathcal V$ of $B$
such tha:t if $X$ is a space and $f, g\co X\to B$ are $\mathcal
V$-close maps, then $f$ is $\mathcal U$-homotopic to $g$ rel
$X_{(f=g)}$.
\end{proposition}

We will need a metric version of Proposition~\ref{prop:classic ANR}
that is only valid for certain triangulated metric spaces $B$. For
the proof of the metric version, we will need the following relative
version of Proposition~\ref{prop:classic ANR} in the compact case.
It is a fairly routine application of Proposition~\ref{prop:classic
ANR} and the Estimated Homotopy Extension Theorem of Chapman--Ferry
\cite{ChapmanFerry}, but we include a proof for completeness.

\begin{lemma}
\label{lem:compact ANR} Let $(Y,d)$ be a compact, metric ANR. For
every $\epsilon >0$ there exists $\delta > 0$ such that: if $X$ is a
space with a closed subspace $X_0$ and $f,g\co X\to Y$ are
$\delta$-close maps for which there is a $\delta$-homotopy $H\co
f|X_0\simeq g|X_0$, then there exists an $\epsilon$-homotopy
$\widetilde{H}\co f\simeq g$ such that $\widetilde{H}|X_0\times I =
H$.
\end{lemma}

\begin{proof}
Let $\epsilon >0$ be given. By Proposition~\ref{prop:classic ANR}
there exists $\mu >0$ such that $\mu <\epsilon$ and if $X$ is any
space, $f,g\co X\to Y$ are $\mu$-close maps, then there exists an
$\epsilon/2$-homotopy $f\simeq g$ rel $X_{(f=g)}$. Choose $\delta
>0$ such that $\delta<\mu/2$. Now suppose $X$ is a space with a
closed subspace $X_0$ and $f,g\co X\to Y$ are $\delta$-close maps
for which there is a $\delta$-homotopy $H\co f|X_0\simeq g|X_0$. By
the Estimated Homotopy Extension Theorem \cite[Proposition
2.1]{ChapmanFerry} there exists a map $\hat g\co X\to Y$ and a
$\delta$-homotopy $\hat H\co f\simeq\hat g$ such that $\hat
H|(X_0\times I)= H$. Thus, $\hat g=\hat H_1$ is $\delta$-close to
$f$, which in turn is $\delta$-close to $g$. Thus, $\hat g$ is
$\mu$-close to $g$. By the choice of $\mu$ there is an
$\epsilon/2$-homotopy $G\co \hat g\simeq g$ rel $X_{(\hat g =g)}$.
Note $X_0\subseteq X_{(\hat g=g)}$.

The concatenation
\[
\hat H\ast G: X\times [0,2] \longrightarrow Y; \quad
(x,t) \longmapsto \begin{cases} \hat H(x,t) & \text{if $0\leq t \leq 1$}\\
                                        G(x,t-1)    & \text{if $1\leq t \leq 2$}
                                        \end{cases}
\]
is an $\epsilon$-homotopy, which can be re-parameterized to give the
homotopy $\widetilde{H}$ as follows. Let $u\co X\to [0,1]$ be a map
such that $u^{-1}(0)= X_0$. Define
\[
q\co X\times [0,2] \longrightarrow X\times [0,1]; \quad
(x,t) \longmapsto \begin{cases} \left(x,t\left( 1-\frac{1}{2}u(x)\right)\right) & \text{if } 0\leq t \leq 1\\
                         \left(x, \frac{1}{2}u(x)t+1-u(x)\right)   & \text{if } 1\leq t \leq 2.
                                        \end{cases}
\]
Then $q$ is a quotient map with the property that each interval $\{
x\}\times [0,1]$ is taken linearly onto $\{ x\} \times [0,
1-\frac{1}{2}u(x)]$  and each interval  $\{ x\}\times [1,2]$ is
taken linearly onto $\{ x\} \times [1-\frac{1}{2}u(x),1]$. The only
non-degenerate point inverses of $q$ are for $(x,1)$ with $x\in
X_0$, in which case $q^{-1}(x,1) = \{x\}\times [1,2]$.  It follows
that $\widetilde{H} := (\hat H\ast G)\circ q^{-1}$ is the desired
homotopy.\qed
\end{proof}

The following result is our metric version of
Proposition~\ref{prop:classic ANR}. It is used in the proof of
Lemma~\ref{lem:equivariant epsilon,MAF}.

\begin{proposition}
\label{prop:metric ANR} Suppose $(B,d)$ is a   metric space
triangulated by a locally finite, finite dimensional simplicial
complex and $A$ is a closed sub-polyhedron of $B$ such that:
\begin{enumerate}
    \item There exists $d_0 >0$ such that: if $v$ and $w$ are distinct vertices of $A$, then $d(v,w)\geq d_0$.
    \item For every $\alpha > 0$  and for every $n\geq 0$,
    there exists $\beta > 0$ such that: if $x$ and $y$ are in
    distinct $n$-simplices of $A$ and $d(x,y)< \beta$, then $x$
    and $y$ are in the $\alpha$-neighborhood of the
    $(n-1)$-skeleton of $A$.
    \item $\Shapes^\triangle(A)$ is finite.
\end{enumerate}
For every $\epsilon > 0$ there exists $\delta >0$ such that: if $X$
is a space and $f, g\co X\to B$ are $\delta$-close maps such that
$f=g$ over $B\setminus A$, then $f$ is $\epsilon$-homotopic to $g$
rel $X_{(f=g)}$.
\end{proposition}

\begin{proof}
The proof is by induction on $\dim B$. If $\dim B= 0$, then choose
$0<\delta<d_0$ (which is independent of $\epsilon$ in this case). It
follows that if $f, g\co X\to B$ are $\delta$-close and $f=g$ over
$B\setminus A$, then $f=g$.

Next assume that $\dim B=n>0$ and the proposition is true in lower
dimensions.

We will now define a particular strong deformation retraction of a
neighborhood of the $(n-1)$-skeleton of $A$ to the $(n-1)$-skeleton.
Let $A^{n-1}$ denote the $(n-1)$-skeleton of $A$ and let $S_n$
denote the set of $n$-simplices of $A$. Since $\Shapes^\triangle(A)$
is finite we can choose a finite subset $T_n$ of $S_n$ such that
each member of $S_n$ is simplicially isometric to a member of $T_n$.
For each $\tau\in T_n$, choose $b_{\tau}\in
\tau\setminus\partial\tau$ and let $r_\tau\co
\tau\setminus\{b_\tau\}\times I\to \tau\setminus\{ b_\tau\}$ be a
strong deformation retraction onto $\partial\tau$. Using the
finiteness of $\Shapes^\triangle(A)$, we can extend the selection of
the points $b_\tau$ to a selection of points
$b_{\sigma}\in\sigma\setminus\partial\sigma$ for every $\sigma\in
S_n$, and we can extend the strong deformation retractions $r_\tau$
to a strong deformation retraction of $B\setminus\{ b_{\sigma} ~|~
\sigma\in S_n\}$ onto $(B\setminus A) \cup A^{n-1}$ so that the
following is true. There is a homotopy
$$r\co B\setminus\{ b_{\sigma} ~|~ \sigma\in S_n\} \times I \to
B\setminus\{ b_{\sigma} ~|~ \sigma\in S_n\}$$
such that:
\begin{enumerate}
    \item $r_0=\id$
    \item $r_t|(B\setminus A)\cup A^{n-1} =\incl$ for all $t\in I$
    \item The image of $r_1$ is $(B\setminus A)\cup A^{n-1}$
    \item $r_t(\sigma\setminus\{ b_\sigma\}) \subseteq \sigma\setminus\{ b_\sigma\}$ for all $t\in I$ and $\sigma\in S_n$
    \item {\bf (Finiteness)} For each $\sigma\in S_n$ there exists $\tau\in T_n$ and a simplicial isometry
    $h\co \sigma\to\tau$ such  that $h(b_\sigma) = b_\tau$ and
    the following diagram commutes for all $t\in I$:
    $$
\begin{diagram}
\node{\sigma\setminus\{ b_\sigma\}} \arrow{s,tb}{r_t|}{} \arrow{e,t}{h|} \node{\tau\setminus\{ b_\tau\}} \arrow{s,b}{(r_\tau)_t} \\
\node{\sigma\setminus\{ b_\sigma\}} \arrow{e,t}{h|} \node{\tau\setminus\{ b_\tau\}}
\end{diagram}
$$
    \end{enumerate}

It follows that $(\ast)$ for every $\gamma >0$ there exists $\rho
>0$ such that:
\begin{enumerate}
    \item If $\sigma\in S_n$, then $b_\sigma\notin \bar{N}_\rho(A^{n-1})$, the closed $\rho$-neighborhood about $A^{n-1}$ in $B$.
    \item If $x,y\in \bar{N}_\rho(A^{n-1})$ and $d(x, y) <\rho$, then $d(r_1(x), r_1(y)) <\gamma$.
    \item For every $x\in \bar{N}_\rho(A^{n-1})$, the track $r(\{x\}\times I)$ has diameter $<\gamma$.
\end{enumerate}

Let $\epsilon >0$ be given. Use Lemma~\ref{lem:compact ANR} and the
assumption that $\Shapes^\triangle(A)$ is finite to choose $\mu>0$
with the following property: if $\Delta$ is any simplex of $A$, $X$
is any space with a closed subspace $X_0$ and $f,g\co X\to\Delta$
are $\mu$-close maps for which there is a $\mu$-homotopy $H\co
f|X_0\simeq g|X_0$, then there exists an $\epsilon/3$-homotopy
$\widetilde{H}\co f\simeq g$ such that $\widetilde{H}|X_0\times I =
H$.

Let $B^{n-1}$ denote the $(n-1)$-skeleton of $B$ and use the
inductive hypothesis to choose $\delta_1>0$ with the following
property: if $X$ is a space and $f, g\co X\to B^{n-1}$ are
$\delta_1$-close maps such that $f=g$ over $B^{n-1}\setminus A$,
then $f$ is $\mu$-homotopic to $g$ rel $X_{(f=g)}$. It follows that
if $X$ is a space and $f,g\co X\to B$ are $\delta_1$-close maps such
that $f^{-1}(B\setminus B^{n-1}) \cup g^{-1}(B\setminus
B^{n-1})\subseteq X_{(f=g)}$, and $f=g$ over $B\setminus A$, then
$f$ is $\mu$-homotopic to $g$ rel $X_{(f=g)}$.

Let $\rho=\rho(\gamma)$ be given by $(\ast)$ above where
$\gamma=\min\{ \delta_1, \epsilon/3\}$. Let $\beta
> 0$ be given by hypothesis (2) in the proposition for
$\alpha=\rho$. Choose $\delta > 0$ such that $\delta<\min\{\delta_1,
\rho, \beta, \mu \}$.

Now suppose given a space $X$ and $\delta$-close maps $f, g\co X\to
B$ such that $f=g$ over $B\setminus A$. We must show that $f$ is
$\epsilon$-homotopic to $g$ rel $X_{(f=g)}$. For each $\sigma\in S_n$, define subspaces
\[\begin{array}{ccccccc}
X^{n-1} &=& X_{(f=g)} \cup \left(f^{-1}(A^{n-1})\cap g^{-1}(A^{n-1})\right)
& \qquad & Y &=& \bigcup_{\sigma\in S_n} X_\sigma \vspace{2mm} \\
X_\sigma &=& f^{-1}(\sigma)\cap g^{-1}(\sigma)
& \qquad & Z &=& X\setminus (Y\cup X^{n-1}).
\end{array}\]
If $x\in Z$, then $f(x), g(x)$ are in distinct $n$-simplices of $A$ and $d(f(x), g(x))
<\delta$. The choice of $\delta <\beta$ implies that $f(x), g(x) \in
\bar{N}_\rho(A^{n-1})$. Define maps
\[
f^{n-1}: X^{n-1}\cup Z \longrightarrow B ;\quad
f^{n-1} = \begin{cases} f & \text{on $X^{n-1}$}\\
                        r_1f & \text{on $Z$}
            \end{cases}
\]
and
\[
g^{n-1}: X^{n-1}\cup Z \longrightarrow B; \quad
            g^{n-1} = \begin{cases} g & \text{on $X^{n-1}$}\\
                        r_1g & \text{on $Z$.}
            \end{cases}
\]
The choice of $\rho$ implies that there are $\epsilon/3$-homotopies
$E\co f|\simeq f^{n-1}$ rel $X^{n-1}$ and $F\co g| \simeq g^{n-1}$
rel $X^{n-1}$. Moreover, the choice of $\rho$ implies that $f^{n-1}$
and $g^{n-1}$ are $\delta_1$-close.
Define
\[
f'\co X \longrightarrow B ; \quad  
f' = \begin{cases} f^{n-1} & \text{on $X^{n-1}\cup Z$} \\
                     f        & \text{ on $Y$}
                     \end{cases}
= \begin{cases} f & \text{on $X^{n-1}\cup Y$}\\
                r_1f & \text{on $Z$}
                \end{cases}
\]
and
\[
g': X\longrightarrow B ; \quad  
g' = \begin{cases} g^{n-1} & \text{on $X^{n-1}\cup Z$} \\
                     g        & \text{ on $Y$}
                     \end{cases}
= \begin{cases} g & \text{on $X^{n-1}\cup Y$}\\
                r_1g & \text{on $Z$.}
                \end{cases}
\]
The homotopies $E$ and $F$ can be extended to
$\epsilon/3$-homotopies $E'\co f\simeq f'$ rel $X^{n-1}\cup Y$ and
$F'\co g\simeq g'$ rel $X^{n-1}\cup Y$.
By the inductive assumption, there exists a $\mu$-homotopy
$$H^{n-1}\co f^{n-1}\simeq g^{n-1} \text{ rel } X_{(f=g)}.$$
The choice of $\mu$ implies that $H^{n-1}$ can be extended to an
$\epsilon/3$-homotopy $H\co f'\simeq g'$. Clearly, $H$ is rel
$X_{(f=g)}$. Finally, we concatenate the three $\epsilon/3$-homotopies $E'$, $H$,
and $F'$ to get an $\epsilon$-homotopy $f\simeq f'\simeq g' \simeq
g$ rel $X_{(f=g)}$.\qed
\end{proof}

\begin{corollary}
\label{cor:FIT metric ANR} Suppose $(B,d)$ is a triangulated, metric
space of finite isometry type. For every $\epsilon > 0$ there exists
$\delta
>0$ such that: if $X$ is a space and $f, g\co X\to B$ are
$\delta$-close maps, then $f$ is $\epsilon$-homotopic to $g$ rel
$X_{(f=g)}$.
\end{corollary}

\begin{proof}
Apply Proposition~\ref{prop:metric ANR}, Definition~\ref{def:FIT},
and Remark~\ref{rem:finite}.\qed
\end{proof}


\subsection{Metric version of Chapman's MAF sucking theorem}\label{subsec:Metric version of Chapman's MAF
sucking theorem}

A fundamental result concerning approximate fibrations defined on
manifolds is the following theorem of Chapman \cite{Chapman}.

\begin{theorem}[Chapman's MAF Sucking Theorem]
\label{Chapman's MAF Sucking Theorem} Suppose $B$ is a locally
compact, separable, metrizable, locally polyhedral space. For each
integer $m > 4$ and each open cover $\alpha$ of $B$, there exists an
open cover $\beta$ of $B$ such that: if $M$ is an $m$-dimensional
manifold and $p\co M\to B$ is a proper $\beta$-fibration, then $p$
is $\alpha$-close to a proper approximate fibration $p'\co M\to B$.
\end{theorem}

It is referred to as ``sucking'' because it says that a map that is
nearly an approximate fibration can be deformed, or sucked, into the
space of approximate fibrations. Chapman had earlier proved a
Hilbert cube manifold sucking theorem \cite{ChapmanTAMS}.

The purpose of this section is to establish a metric version of
Chapman's result in which $B$ is given a fixed metric and the open
covers $\alpha$ and $\beta$ of $B$ are replaced by numbers $\epsilon
>0$ and $\delta >0$, respectively. See
Corollary~\ref{cor:FIT-nonrel} below. In fact, we establish a
relative result in Corollary~\ref{cor:FIT}, A special case of
Corollary~\ref{cor:FIT}, namely Corollary~\ref{cor:chapman}, is the
key result that will be applied in Lemma~\ref{lem:equivariant
epsilon,MAF} in the course of proving an equivariant version of
sucking in Section~\ref{section:sucking}.

Of course, the numbers $\epsilon$ and $\delta$ correspond to open
covers of the metric space $B$ by balls of radius $\epsilon$ and
$\delta$, respectively. Thus, the metric result applies to fewer
situations (because not all open covers consist of balls of fixed
radius), but has a stronger conclusion than Chapman's
Theorem~\ref{Chapman's MAF Sucking Theorem}. Such a variation is not
true in general without further restrictions on the metric space
$B$. 
For example, let $B=\bigsqcup_{i=1}^\infty S_i^1$ be the disjoint union of circles metrized
so that each circle is a subspace of $B$ and $\lim_{i\to\infty} \delta_i = 0$, 
where $\delta_i = \diam (S_i^1)$.
If $M=S^k$ is a single $k$-sphere, where $k > 1$, then there exists no approximate fibration
$M\to B$.
On the other hand, if for each $i=1,2,\dots$, $p_i\co M\to B$ is a map such that
$p_i(M)\subseteq S_i^1$, then $p_i$ is a proper $\delta_i$-fibration.

That there are metric versions of Chapman's theorem is not a
new observation. It was pointed out by Hughes \cite[Remark
7.4]{Hughes1985Q} that such a metric version holds for $B=\br^n$
with the standard Euclidean metric. Hughes--Prassidis
\cite[Footnote, p.~10]{HughesPrassidis} assert the metric result for
``non-compact manifolds with sufficiently homogeneous metrics.''
Hughes--Ranicki \cite[Thm.~16.9]{HughesRanicki} assert and use
Corollary~\ref{cor:euclidean_delta_sucking} in the case $n=1$. In
each of these three references it is claimed that these variations
can be proved by closely examining Chapman's proof. This is indeed
the case; however, detailed explanations  have not heretofore
appeared in the literature. Since we require a yet more general
result, we provide a detailed outline of proof.

The proof of our metric result follows Chapman's papers
\cite{ChapmanTAMS} \cite{Chapman} as well as Hughes \cite{Hughes1981}. The
heart of Chapman's proof consists of his Handle Lemmas (quoted below
as Lemmas~\ref{lem:chapman handle 1} and \ref{lem:chapman handle
2}), which we can use without change. Chapman proves those lemmas by
engulfing and torus geometry (also known as torus tricks). His
methods require high dimensions. What we have written here is just a
careful repackaging of that part of Chapman's proof that comes after
his Handle Lemmas.

We begin by discussing a limit result (Lemma~\ref{lem:limit})
implicit in Chapman's work \cite{ChapmanTAMS} \cite{Chapman}. The proof
requires the following result, the proof of which is based on Coram
and Duvall \cite[Proposition 1.1]{CoramDuvall}.

\begin{lemma}
\label{lem:CD-like} Suppose $(B,d)$ is a metric ANR and $U$ is an
open subset of $B$. For every $\mu > 0$,  for every compact metric
space $Z$, and for every homotopy $F\co Z\times I\to U$, there
exists $\nu >0$ such that the following holds: if $\epsilon > 0$,
$E$ is an ANR, $p\co E\to B$ is a proper $\epsilon$-fibration over
$U$, and $f\co Z\to E$ is a map with $pf$ $\nu$-close to $F_0$, then
there is a homotopy $\widetilde{F}\co Z\times I\to E$ such that
$\widetilde{F}_0 = f$ and $p\widetilde{F}$ is $(\epsilon
+\mu)$-close to $F$.
\end{lemma}

\begin{proof}
Given $\mu >0$, a compact metric space $Z$, and a homotopy $F\co
Z\times I\to U$, choose a compact neighborhood $K$ of $F(Z\times\{
0\})$ with $K\subseteq U$. Choose $d_0 >0$ such that the
$d_0$-neighborhood of $F(Z\times\{ 0\})$ is contained in $K$. Let
$\delta=\min\{ d_0,\mu/2\}$. Choose $\nu>0$ such that any two
$\nu$-close maps into $K$ are $\delta$-homotopic in $U$ (see
Proposition~\ref{prop:classic ANR}).

Now suppose given $\epsilon>0$, an ANR $E$, a map $p\co E\to B$ that
is a proper $\epsilon$-fibration over $U$, and a map $f\co Z\to E$
such that $pf$ is $\nu$-close to $F_0$. Let $\co Z\times [-1, 0]\to
U$ be a $\delta$-homotopy such that $J_{-1}=f$ and $J_0=F_0$. Thus,
the image of $J$ is contained in $K\subseteq U$. Define
\[
\Phi:
Z\times [-1,1] \longrightarrow U; \quad
(z,t) \longmapsto \begin{cases} J(z,t) & \text{if $-1\leq t\leq 0$}\\
                          F(z,t) & \text{if $0\leq t\leq 1$}.
                          \end{cases}
\]
It follows that there exists a lift $\widetilde{\Phi}\co Z\times
[-1,1]\to E$ such that $\widetilde{\Phi}_0 = f$ and
$p\widetilde{\Phi}$ is $\epsilon$-close to $\Phi$. 
Since $Z$ is compact, there exists $q\in (0,1)$ such that
$F(\{ z\}\times [0,q)])$ has diameter
less than $\mu/2$ for every $z\in Z$. Finally, define
\[
\widetilde{F}: Z\times I \longrightarrow E ;\quad
(z,t) \longmapsto \begin{cases}
                                                        \widetilde{\Phi}(z, 2t/q -1) & \text{ if $ 0 \leq t \leq q/2 $}\\
                                                        \widetilde{\Phi}(z, 2t-q)   & \text{ if $ q/2 \leq t \leq q $}\\
                                                        \widetilde{\Phi}(z,t)          & \text{ if $ q  \leq t \leq 1 $}
                                                \end{cases}
\]
One may check that $\widetilde{F}_0 = f$ and that $p\widetilde{F}$
is $(\epsilon +\mu)$-close to $f$.\qed
\end{proof}

\begin{lemma}[Limit Lemma]
\label{lem:limit} Suppose $E$ and $B$ are ANRs, $\mathcal U$ is a
collection of open subsets of $B$, $\{\epsilon_i\}_{=1}^\infty$ is a
sequence of positive numbers with $\lim_{i\to\infty}\epsilon_i=0$,
and there are proper maps $q,q_i\co E\to B$ for $i=1,2,3,\dots$ with
$\lim_{i\to\infty}q_i=q$ (uniformly). If $q_i$ is an
$\epsilon_i$-fibration over $U$ for each $U\in\mathcal U$, then $q$
is an approximate fibration over $\cup\mathcal U$.
\end{lemma}

\begin{proof}

It suffices to show that given $U\in\mathcal U$, $q|\co q^{-1}(U)\to
U$ is an approximate fibration for the class of compact metric
spaces. For then results of Coram and Duvall \cite[Theorem
2.6]{CoramDuvall2}, \cite[Uniformization, page 43]{Coram} imply that
$q|\co q^{-1}(\cup\mathcal U)\to\cup\mathcal U$ is an approximate
fibration. Thus, suppose given $\epsilon > 0$, a compact metric
space $Z$, and a homotopy lifting problem
$$
\begin{diagram}
\node{Z} \arrow{s,tb}{\times 0}{} \arrow{e,t}{f} \node{q^{-1}(U)} \arrow{s,b}{q|} \\
\node{Z\times I}\arrow{ne,t,..}{\tilde F} \arrow{e,t}{F} \node{U}
\end{diagram}
$$
For $\mu=\epsilon/3$, let $\nu=\nu(U,\mu,Z,F)$ be given by
Lemma~\ref{lem:CD-like}. Now choose $i\in\bn$ such that $q_i$ is
$\nu$-close to $q$ and $q_i$ is an $\epsilon/3$-fibration over $U$.
It follows that $q_i f$ is $\nu$-close to $qf=F_0$. Thus,
Lemma~\ref{lem:CD-like} implies there exists $\widetilde{F}\co
Z\times I\to E$ such that $q_i\widetilde{F}$ is $2\epsilon/3$-close
to ${F}$. It follows that $q\widetilde{F}$ is $\epsilon$-close to
$F$.\qed
\end{proof}

We next quote the two handle lemmas of Chapman \cite[Lemma 5.1,
Theorem 5.2]{Chapman}.

\begin{lemma}[Chapman's First Handle Lemma]
\label{lem:chapman handle 1} Suppose $k$ is a positive integer and
$\br^k \hookrightarrow B$ is an open embedding, where $B$ is an ANR.
For every $m > 4$ and $\epsilon >0$ there is exists a $\delta >0$
such that: if $\mu>0$, $M$ is an $m$-manifold, and $p\co M\to B$ is
a proper map that is a $\delta$-fibration over $\mathrm{B}_3^k$,
then there is a proper map $p'\co M\to B$ such that
\begin{enumerate}
    \item $p'$ is a $\mu$-fibration over $\mathrm{B}_1^k$,
    \item $p'$ is $\epsilon$-close to $p$,
    \item $p=p'$ on $M\setminus p^{-1}(\ob{}_3^k)$. \qed
\end{enumerate}
\end{lemma}

\begin{lemma}[Chapman's Second Handle Lemma]
\label{lem:chapman handle 2} Suppose $k$ is a nonnegative integer
and $\oc(X)\times\br^k \hookrightarrow B$ is an open embedding,
where $B$ is an ANR and $X$ is a compact ANR. For every $m > 4$ and
$\epsilon
>0$ there is exists a $\delta >0$ such that: if $\mu>0$ there exists
$\nu>0$ so that the following statement is true:
\newline
if  $M$ is a $m$-manifold and $p\co M\to B$ is a proper map that is
a $\delta$-fibration over $c_3(X)\times\mathrm{B}_3^k$ and a
$\nu$-fibration over
$[c_3(X)\setminus\oc_{1/3}(X)]\times\mathrm{B}_3^k$, then there is a
proper map $p'\co M\to B$ such that
\begin{enumerate}
    \item $p'$ is a $\mu$-fibration over $c_1(X)\times\mathrm{B}_1^k$,
    \item $p'$ is $\epsilon$-close to $p$,
    \item $p=p'$ on $M\setminus p^{-1}(\oc_{2/3}(X)\times\ob{}_3^k)$. \qed
\end{enumerate}
\end{lemma}

\begin{remark} These two handle lemmas are not an exact quote of Chapman \cite[Lemma 5.1, Theorem 5.2]{Chapman}, but the difference is
insignificant. Chapman considers maps directly to $\br^k$ and
$\oc(X)\times\br^k$, rather than to manifolds in which these spaces
are embedded. The lemmas above are formal, immediate consequences of
Chapman's lemmas.
\end{remark}

\begin{remark}
It is important to note that both of these Handle Lemmas are
independent of the metric on $B$. That is, the various constants
$\delta$ and $\nu$ depend on the metric, but their existence is
independent of the metric. This is because they depend only on the
metric on a compact portion of $B$.
\end{remark}

\begin{hypothesis}\label{hyp:finite}
The following list of technical hypotheses are used later in this section.
\begin{enumerate}
    \item Suppose $B$ be a locally finite polyhedron and let $A$ be a closed sub-polyhedron of dimension $n$.
    \item Fix a locally finite  triangulation of $B$ with respect to which $A$ is triangulated as a closed subcomplex.
    We will abuse notation and make no distinction between a
    simplicial complex and its underlying polyhedron.
    \item Suppose $d$ is a  metric for $B$ compatible with the topology on $B$.
    \item Let $\mathcal B$ be the set of barycenters of simplices in $A$. For each $0\leq k\leq n$, let ${\mathcal B}_k =\{ b\in{\mathcal B} ~|~ b
    \text{ is the barycenter of a simplex of $A$ of dimension
    $k$}\}$. For each $b\in\mathcal B$, let $\sigma_b$ denote
    the simplex of $A$ of which $b$ is the barycenter.
    \item For each $b\in\mathcal B$, fix an open neighborhood $V_b$ of $b$ in $B$, a compact polyhedron $X_b$ and a homeomorphism
    $\phi_b\co\oc(X_b)\times\br^k\to V_b$, where $b\in\mathcal
    B_k$. If $k=n$, then $X_b=\varnothing$ and $\oc(X_b)$ is a
    single point.
  \item Assume that $V_{b_1}\cap V_{b_2}=\varnothing$ whenever $0\leq k\leq n$ and $b_1, b_2\in\mathcal B_k$.
  \item For each $b\in\mathcal B$, let $C_b$ denote the closed star neighborhood of $b$ in the second barycentric subdivision of $B$.
  Thus, $A\subseteq \bigcup\{ C_b ~|~ b\in\mathcal B \}$.
  \item For each $0\leq k\leq n$ and $b\in\mathcal B_k$, assume $C_b = \phi_b\left(c_1(X_b)\times\mathrm{B}_1^k\right)$ and that
  $\phi_b\left(\{ v\}\times\br^k\right)$ is a neighborhood of
  $b$ in $\sigma_b$, where $v$ is the cone point of $\oc(X_b)$.
  \item Let $V=\bigcup\{ V_b ~|~ b\in {\mathcal B}\}$.
  \item The metric $d$ restricts to a complete metric on the closure of $V$.
  \item For each $0\leq k\leq n$ and $b\in\mathcal B_k$, let $W_b = \phi_b\left(c_{1.1}(X_b)\times\mathrm{B}_{1.1}^k\right)$.
  \item Choose numbers $1.2 < r_0 < r_1< \cdots < r_n =1.3$ and assume that for each $0\leq k< n$ and $b\in\mathcal B_k$, we have
\begin{enumerate}
    \item $\phi_b\left(\left[c_3(X_b)\setminus\oc_{1/3}(X_b)\right]\times B_3^k\right) \subseteq
    \bigcup\{ \phi_a\left(c_{r_{k+1}}(X_a)\times
    B_{r_{k+1}}^\ell\right) ~|~ k+1\leq\ell\leq n ,
    a\in\mathcal B_\ell\},$
    \item $\phi_b\left(c_{2/3}(X_b)\times B_3^k\right)$ misses
     $\bigcup\{ \phi_a\left(c_{r_{k}}(X_a)\times
     B_{r_{k}}^\ell\right) ~|~ k+1\leq\ell\leq n ,
     a\in\mathcal B_\ell\}.$
\end{enumerate}
  \item {\bf (Finiteness)} For each $0\leq k\leq n$ and $b\in\mathcal B_k$, let $d_b=\phi_b^\ast d$ be the metric on
  $\oc(X_b)\times\br^k$ obtained by pulling back $d$ along
  $\phi_b$. For each $0\leq k\leq n$, assume that $\{ d_b ~|~
  b\in\mathcal B_k\}$ is finite.
    \end{enumerate}
\end{hypothesis}

\begin{remark} The Finiteness condition above says in the first place that the set $\{ X_b ~|~ b\in \mathcal B\}$ of non-isomorphic polyhedra that are links of barycenters   is finite. In the second place, it says that for any given polyhedron $X$ occurring as a link and any $0\leq k\leq n$, even though there might be infinitely many different open embeddings given of $\oc(X)\times\br^k$ into $B$, there are only finitely many different induced metrics on $\oc(X)\times\br^k$.
\end{remark}

\begin{remark} Note  that given condition (1) in Hypothesis~\ref{hyp:finite}, conditions (2) through (12) may always be achieved. They are listed to fix notation. Thus, condition (13) is the only extra assumption.
\end{remark}

The proof of the next result is based on Hughes \cite[Lemma
10.1]{Hughes1981}, which in turn is based on Chapman \cite[Section
6]{ChapmanTAMS}.

\begin{proposition}
\label{prop:Chapman main} Assume Hypothesis~\ref{hyp:finite}. For
each integer $m > 4$ and each $\epsilon >0$ there exists a $\delta
>0$ such that for every $\mu >0$ if $M$ is an $m$-dimensional
manifold and $p\co M\to B$ is a proper $\delta$-fibration over $V_b$
for each $b\in\mathcal B$, then $p$ is $\epsilon$-close to a proper
map $p'\co M\to B$ such that $p=p'$ on $M\setminus p^{-1}(V)$ and
$p'$ is a $\mu$-fibration over $W_b$ for each $b\in \mathcal B$.
\end{proposition}

\begin{proof}
Let $\epsilon > 0$ be given. Inductively define small positive
numbers
$$\epsilon_0, \delta_0, \epsilon_1, \delta_1,\dots, \epsilon_n, \delta_n$$
with the following properties:
\begin{enumerate}
    \item $0 < \epsilon_0 < \epsilon/(n+1)$,
    \item $\delta_k <\delta(\epsilon_k)$, where $\delta(\epsilon_k)$ is given by the Handle Lemma~\ref{lem:chapman handle 1} (if $k=n$) or
    \ref{lem:chapman handle 2} (if $k<n$) for the open
    embeddings $\phi_b\co\br^n\hookrightarrow B$ or
    $\phi_b\co\oc(X_b)\times\br^k\hookrightarrow B$ for each
    $b\in\mathcal B_k$ (The handle lemmas are applied
    independently for each   $b\in\mathcal B_k$. Since $\mathcal
    B_k$ may be infinite, the Finiteness condition of
    Hypothesis~\ref{hyp:finite} is crucial at this step.),
    \item $\delta_k < \delta_{k-1}/2$,
    \item $\epsilon_k <\epsilon/(n+1)$,
    \item For each $b\in\mathcal B_k$, any map to $B$ that is $\epsilon_k$-close to a $(\delta_{k-1}/2)$-fibration over $\phi_b\left(\br^n\right)$
    or $\phi_b\left(c_3(X_b)\times\mathrm{B}_3^k\right)$ is
    itself a $\delta_{k-1}$-fibration over
    $\phi_b\left(\br^n\right)$ or
    $\phi_b\left(c_3(X_b)\times\mathrm{B}_3^k\right)$,
    respectively. (The Finiteness condition is again being used
    here.)
\end{enumerate}
Set $\delta=\delta_m$ and let $\mu>0$ be given. let  $p\co M\to B$
be given as in the hypothesis. We will produce a map $p'\co M\to B$
that is a $\mu$-fibration over each $W_b$. It suffices to construct
a sequence of maps $p=p^{n+1}, p^n,\dots,p^1,p^0=p'$ such that $p^k$
is $\epsilon_k$-close to $p^{k+1}$ and $p^k$ is a $\mu$-fibration
over $ \phi_b\left(c_{r_{k}}(X_b)\times B_{r_{k}}^\ell\right)$ for
$k\leq\ell\leq n$ and $b\in\mathcal B_\ell\}.$ First, inductively
define small positive numbers $\nu_{-1}, \nu_0,\dots, \nu_n$ be
setting $\nu_{-1}=\mu$ and for $k=0,\dots, n-1$, choosing $\nu_k <
\mu$ such that $\nu_k < \nu(\nu_{k-1})$, where $\nu(\nu_{k-1})$ is
given by the Handle Lemma~\ref{lem:chapman handle 2} for the open
embeddings $\phi_b: \oc(X_b)\times\br^k\hookrightarrow B$ (The
Finiteness condition of Hypothesis~\ref{hyp:finite} is used here to
apply the handle lemma independently for each  $b\in\mathcal B_k$.)

Using the appropriate Handle Lemma, we inductively produce the maps
$p^k$ (starting with $k=n$) so that
\begin{enumerate}
    \item $p^k$ is a $\nu_{k-1}$-fibration over
    $\phi_b\left(c_{r_{k}}(X_b)\times B_{r_{k}}^\ell\right)$ for
    each $b\in\mathcal B_{\ell}$,
    \item $p^k$ is $\epsilon_k$-close to $p^{k+1}$,
    \item $p^k=p^{k+1}$ over $B\setminus\bigcup_{b\in\mathcal B_k}\left[\oc_{2/3}(X_b)\times\ob{}_3^k\right]$.
\end{enumerate}
In order to apply the Handle Lemma inductively simply observe that
$p^k$ is a $\delta_{k-1}$-fibration over
$\phi_b\left(c_3(X_b)\times\mathrm{B}_3^k\right)$. Also observe that
$p^k$ is a $\mu$-fibration over $\phi_b\left( c_{r_{k}}(X_b)\times
B_{r_{k}}^\ell\right)$ for each $k\leq\ell\leq n$ and $b\in\mathcal
B_\ell$.\qed
\end{proof}

The next corollary is essentially a  renaming of some of the sets in
Proposition~\ref{prop:Chapman main}. As such, it can be viewed as a
corollary to the proof of Proposition~\ref{prop:Chapman main}.
However, a more formal derivation is also given. We begin by
introducing some more notation.

\begin{notation}
Assume Hypothesis~\ref{hyp:finite}. For each $R>1$, $0\leq k\leq n$,
and $b\in\mathcal B_k$, define
\[
V_b^R := \phi_b\left(\oc_R(X_b)\times\ob{}_R^k\right)\subseteq B. 
\]
\end{notation}

\begin{corollary}
\label{cor:Chapman main} Assume Hypothesis~\ref{hyp:finite}. For
each integer $m > 4$, each $1<R_2<R_1$, and each $\epsilon >0$ there
exists a $\delta=\delta(m,\epsilon, R_1, R_2) >0$ such that for
every $\mu >0$ if $M$ is an $m$-dimensional manifold and $p\co M\to
B$ is a proper $\delta$-fibration over $V_b^{R_1}$ for each
$b\in\mathcal B$, then $p$ is $\epsilon$-close to a proper map
$p'\co M\to B$ such that $p=p'$ on $M\setminus \bigcup_{b\in\mathcal
B}p^{-1}(V_b^{R_1})$ and $p'$ is a $\mu$-fibration over $V_b^{R_2}$
for each $b\in \mathcal B$.
\end{corollary}

\begin{proof}
Fix a homeomorphism $h\co [0,R_1)\to [0,\infty)$ such that $h(t) =
t$ for all $0\leq t \leq 1$ and $h(R_2) = 1.1$. For each $0\leq
k\leq n$ and $b\in\mathcal B_k$, define a homeomorphism $h_b\co
\oc_{R_1}(X_b)\times\ob{}_{R_1}^k\to \oc(X_b)\times\br^k$ by
$h_b([x,t],y) = ([x,h(t)], (h(y_1), h(y_2),\dots, h(y_k))$ for all
$x\in X_b$, $t\in[0,R_1)$, and $y=(y_1,y_2,\dots,
y_k)\in\ob{}_{R_1}^k$. Apply Proposition~\ref{prop:Chapman main} to
the open embeddings $\phi_b\circ h_b^{-1}$.\qed
\end{proof}

Note that in the statement of Corollary~\ref{cor:Chapman main},
$\delta=\delta(m,\epsilon, R_1, R_2)$ also depends on the set-up in
Hypothesis~\ref{hyp:finite}, but we suppress that dependence in the
notation.

\begin{theorem}
\label{thm:metric sucking} Assume Hypothesis~\ref{hyp:finite}. For
each integer $m > 4$ and $\epsilon >0$ there exists $\delta >0$ such
that: if $M$ is an $m$-dimensional manifold and $p\co M\to B$ is a
proper $\delta$-fibration over $V$, then $p$ is $\epsilon$-close to
a proper map $p'\co M\to B$ such that $p=p'$ on $M\setminus
p^{-1}(V)$ and $p'$ is an approximate fibration over $A$.
\end{theorem}

\begin{proof}
For each $i=1,2,3,\dots$, let $R_i= 2+\frac{1}{i}$. Let $m > 4$ and
$\epsilon >0$ be given. For each $i=1,2,3,\dots$ let
$\delta_i=\delta(m,\epsilon/2^i, R_i, R_{i+1}) > 0$ be given by
Corollary~\ref{cor:Chapman main}. We may also assume that $\delta_i
< 1/i$ so that $\lim_{i\to\infty}\delta_i = 0$. Let
$\delta=\delta_1$ and suppose $M$ is an $m$-dimensional manifold and
$p\co M\to B$ is a proper $\delta$-fibration over $V$. Use
Corollary~\ref{cor:Chapman main} to define inductively a sequence of
proper maps $q_i\co M\to B$, $i=1,2,3,\dots$, such that:
\begin{enumerate}
    \item $q_1=p$,
    \item $q_i$ is $(\epsilon/2^i)$-close to $q_{i+1}$,
    \item $q_i$ is a $\delta_i$-fibration over $V_b^{R_i}\supseteq V_b^2$ for each $b\in\mathcal B$,
    \item $q_i=q_{i+1}$ on $M\setminus\bigcup_{b\in\mathcal B}p^{-1}\left(V_b^{R_i}\right)$.
\end{enumerate}
The completeness of the metric on the closure of $V$ implies that
the uniform limit $p'=\lim_{i\to\infty}q_i$ exists. Clearly, $p'$ is
$\epsilon$-close to $p$ and $p=p'$ on
$M\setminus\bigcup_{b\in\mathcal B}p^{-1}\left(V_b^{R_1}\right)
\supseteq M\setminus p^{-1}\left(V\right)$. By Lemma~\ref{lem:limit},
$p'$ is proper and an approximate fibration over
$\bigcup_{b\in\mathcal B}V_b^2\subseteq A$.\qed
\end{proof}

\begin{corollary}
\label{cor:FIT} Suppose $(B,d)$ is a  metric space triangulated by a
locally finite simplicial complex, $A$ is a closed sub-polyhedron of
$B$, $\Shapes^\boxtimes(B, A)$ is finite, $U$ is an open subset of
$B$ containing $A$, and the metric $d$ is complete on the closure of
$U$. For every integer $m > 4$ and $\epsilon > 0$, there exists a
$\delta >0$ such that: if $M$ is an $m$-dimensional manifold and
$p\co M\to B$ is a proper $\delta$-fibration over $U$, then $p$ is
$\epsilon$-close to a proper map $p'\co M\to B$ such that $p=p'$ on
$M\setminus p^{-1}(U)$ and $p'$ is an approximate fibration over
$A$.
\end{corollary}

\begin{proof}
This follows directly from Theorem~\ref{thm:metric sucking}.\qed
\end{proof}

The following corollary follows immediately from
Corollary~\ref{cor:FIT} by taking $A=B$. It is the metric version of
Chapman's MAF Sucking Theorem~\ref{Chapman's MAF Sucking Theorem}.

\begin{corollary}[Metric MAF Sucking]
\label{cor:FIT-nonrel} Suppose $(B,d)$ is a  complete metric space
triangulated by a locally finite simplicial complex such that
$\Shapes^\boxtimes(B)$ is finite. For every integer $m > 4$ and
$\epsilon
> 0$, there exists a $\delta >0$ such that: if $M$ is an
$m$-dimensional manifold and $p\co M\to B$ is a proper
$\delta$-fibration, then $p$ is $\epsilon$-close to a manifold
approximate fibration $p'\co M\to B$. \qed
\end{corollary}

\begin{corollary}
\label{cor:new FIT-nonrel} Suppose $(B,d)$ is a triangulated metric
space of finite isometry type.
 For every integer $m > 4$ and
$\epsilon > 0$, there exists a $\delta >0$ such that: if $M$ is an
$m$-dimensional manifold and $p\co M\to B$ is a proper
$\delta$-fibration, then $p$ is $\epsilon$-homotopic to a manifold
approximate fibration $p'\co M\to B$.
\end{corollary}

\begin{proof}
Apply Corollary~\ref{cor:FIT-nonrel}, Corollary~\ref{cor:FIT metric
ANR}, and  Definition~\ref{def:FIT}.\qed
\end{proof}

The following corollary is the exact form of metric MAF sucking that
we will use in Lemma~\ref{lem:equivariant epsilon,MAF} below in the
course of proving an equivariant version of MAF sucking.

\begin{corollary}
\label{cor:chapman} Suppose  $Y$ is a compact metric space and $d$
is a complete metric for the open cone $\oc(Y)$.   Restrict the
metric to the subset $B := Y\times (0,\infty)$. Select
$0<r_1<r_2<r_3$ and an integer $m > 4$. Suppose $B$ is triangulated
by a locally finite simplicial complex such that there is a closed
polyhedral neighborhood $A$ of $Y\times [r_3,\infty)$ contained in
$Y\times (r_2,\infty)$ with $\Shapes^\boxtimes(B,A)$ finite. For
every $\epsilon >0$ there exists $\delta >0$ such that: if $P$ is an
$m$-dimensional manifold and $p\co P\to B$ is a proper
$\delta$-fibration over $Y\times (r_1,\infty)$, then $p$ is
$\epsilon$-close to a map $p'$ that is a MAF over
$Y\times(r_3,\infty)$ such that $p=p'$ on $p^{-1}(Y\times (0,r_2])$.
\end{corollary}

\begin{proof} Apply Corollary~\ref{cor:FIT} with $B=Y\times (0,\infty)$ and $U=Y\times (r_2,\infty)$.\qed
\end{proof}

As mentioned above, the following result has been used in the
literature (e.g., \cite[Theorem 16.9]{HughesRanicki}) , but a
detailed derivation has heretofore not appeared.

\begin{corollary}\label{cor:euclidean_delta_sucking}
Suppose $\br^n$ is given its standard metric. For every $m > 4$ and
every $\epsilon > 0$ there exists $\delta=\delta(n,m,\epsilon) > 0$
such that: if $M$ is an $m$-manifold and $p\co M\to\br^n$ is a
proper $\delta$-fibration, then $p$ is $\epsilon$-homotopic to a
MAF. Consequently, any proper bounded fibration $p: M \to \br^n$ is
boundedly homotopic to a MAF.
\end{corollary}

\begin{proof}
For the first statement, apply Corollary~\ref{cor:FIT} with $\br^n=
A=U=B$ and triangulate $\br^n$ so that
$\Shapes^\boxtimes(\br^n)$ is finite. For the second statement, apply a
standard scaling procedure; see Hughes--Ranicki \cite[Corollary
16.10]{HughesRanicki} and the proof of
Corollary~\ref{cor:equivariant_bounded_sucking} below.\qed
\end{proof}

\section{Orthogonal actions}\label{section:orthogonal}

In this section we establish various lifting properties for maps
associated to actions of certain finite subgroups $G$ of the
orthogonal group $O(n)$. For our Main Theorem~\ref{intro thm:main},
we only need the case of $C_2$ acting by reflection on $\br$.
However, for the proof of the more widely applicable
Theorem~\ref{intro thm: ortho sucking} (Orthogonal Sucking),  we
need to work in a more general context.

We shall assume the following notation throughout the remainder of
the paper.

\begin{notation} Suppose $G$ is a finite subgroup of the
orthogonal group $O(n)$. We assume that $\br^n$ has the Euclidean
metric. Write $X := S^{n-1}/G$. Since $G$ acts on $\br^n$ by
isometries, we may endow the open cone $\oc(X) = \br^n/G$ with the
quotient metric. Then note that the quotient map
$q_{\br^n}\co\br^n\to\oc(X)$ is distance non-increasing.
\end{notation}

At certain specified times, we shall assume a freeness hypothesis,
as follows.

\begin{hypothesis}\label{hypothesis:orthogonal_free}
Suppose $G$ is a finite subgroup of $O(n)$. Furthermore:
\begin{enumerate}
\item
Assume that $G$ acts freely on $S^{n-1}$. Then $X$ is a closed,
smooth manifold of dimension $n-1$. Suppose $M$ is a $G$-space.
The quotient map is denoted by $q_M\co M\to N := M/G$.

\item
Consider $N$ to be a stratified space with exactly one stratum
(itself).  Consider $\oc(X)$ to be a stratified space with
exactly two strata: the cone point $\{ v\}$ and the complement
$\oc(X)\setminus\{ v\} = X\times (0,\infty)$.
\end{enumerate}
\end{hypothesis}

\emph{For example, we assume
Hypothesis~\ref{hypothesis:orthogonal_free} for the remainder of
Section~\ref{section:orthogonal}.}

In the following lemma, we consider $\br^n$ to be a stratified space
with exactly one stratum. However, the lemma holds equally
well---with no change in the proof---if $\br^n$ has exactly two
strata: the origin $\{0\}$ and $\br^n\setminus\{ 0\}$. The lemma is
a special case of a more general theorem of A.~Beshears
\cite[Thm.~4.6]{Beshears}, but it is quite elementary in the case at
hand; therefore, we include a proof.

\begin{lemma}
\label{lemma:basic quotient} The quotient map $q_{\br^n}\co\br^n\to
\oc(X)$ is a stratified fibration.
\end{lemma}

\begin{proof}
Consider a stratified homotopy lifting problem:
$$
\begin{diagram}
\node{Z} \arrow{s,l}{\times 0} \arrow{e,t}{f} \node{\br^n} \arrow{s,r}{q_{\br^n}} \\
\node{Z\times I}\arrow{ne,t,..}{\tilde F} \arrow{e,t}{F} \node{\oc(X)}
\end{diagram}
$$
The stratified condition is equivalent to saying: if $(z,t)\in
Z\times I$, then $F(z,t)= v$ if and only if $F(z,0)=v$ if and only
if $f(z) = 0$. Let $Z_0= f^{-1}(0)$. Then the stratified homotopy
lifting problem above restricts to the following homotopy lifting
problem:
$$\begin{CD}
Z\setminus Z_0 @>{f|}>> \br^n\setminus\{ 0\} \\
@V{\times 0}VV   @VV{q_{\br^n}|}V \\
(Z\setminus Z_0)\times I @>{F|}>> \oc(X)\setminus \{ v\}\\
\end{CD}
$$
Since $q_{\br^n}|\co \br^n\setminus\{ 0\} \to \oc(X)\setminus \{
v\}$ is a fibration (in fact, a covering map), this later problem
has a solution. That is, there is a map $\hat{F}\co (Z\setminus
Z_0)\times I\to \br^n\setminus\{ 0\}$ such that $\hat{F}(z,0) =
f(z)$ and $q_{\br^n}\hat{F}(z,t)=F(z,t)$ for all $(z,t)\in
(Z\setminus Z_0)\times I$. Define
\[
\tilde{F}\co Z\times I\longrightarrow \br^n; \quad
(z,t) \longmapsto \begin{cases} \hat{F}(z,t) & \text{if $z\in Z\setminus Z_0$}\\
                                0 & \text{if $z\in Z_0$}.
                                \end{cases}
\]
Clearly, $\tilde{F}$ is the required stratified solution if
$\tilde{F}$ is continuous. To verify continuity it suffices to
consider a ball $B$ centered at the origin in $\br^n$ and observe
that $\tilde{F}^{-1}(B)$ is open because $q_{\br^n}B$ is open in
$\oc(X)$ and $\tilde{F}^{-1}(B) = F^{-1}(q_{\br^n}B)$ (using the
fact that $q_{\br^n}^{-1}q_{\br^n}B = B$).\qed
\end{proof}

\begin{proposition}
\label{prop:equivariant implies stratified} For every $\delta >0$
and every proper $G$-$\delta$-fibration $p\co M\to\br^n$, the induced map
$p/G\co N\to\oc(X)$ is a proper stratified $\delta$-fibration.
\end{proposition}

\begin{proof}
For reference, we note that the following diagram commutes:
$$\begin{CD}
M @>{q_M}>> N \\
@V{p}VV  @VV{p/G}V \\
\br^n @>{q_{\br^n}}>> \oc(X)
\end{CD}$$
Consider a stratified homotopy lifting problem:
$$
\begin{diagram}
\node{Z} \arrow{s,l}{\times 0} \arrow{e,t}{f} \node{N} \arrow{s,r}{p/G} \\
\node{Z\times I}\arrow{ne,t,..}{\tilde F} \arrow{e,t}{F} \node{\oc(X)}
\end{diagram}
$$
Form the pull-back diagram
$$\begin{CD}
\hat Z @>{\hat f}>> M \\
@V{\hat q}VV @VV{q_M}V \\
Z @>{f}>> N \\
\end{CD} $$
There is a natural $G$-action on $\hat Z$ ($g(z,x) = (z,gx)$) and
$\hat f$ is a $G$-map. Note that there is a stratified homotopy
lifting problem:
$$
\begin{diagram}
\node{\hat{Z}} \arrow{s,l}{\times 0} \arrow[2]{e,t}{p\hat{f}} \node[2]{\br^n} \arrow{s,r}{q_{\br^n}} \\
\node{Z\times I}\arrow{ene,t,..}{H} \arrow[2]{e,t}{F(\hat{q}\times\id_I)} \node[2]{\oc(X)}
\end{diagram}
$$
Lemma~\ref{lemma:basic quotient} implies that there is a stratified
solution $H\co \hat Z\times I\to \br^n$. Now we claim that $H$ is a
$G$-homotopy. This is essentially true because the action of $G$ on
$\br^n$ is free away from the origin and constant on the origin
(thus, we are using the fact that the action of $G$ on $S^{n-1}$ is free). In
more detail, first note that $F(\hat{q}\times\id_I)(\hat{z},t)=
F(\hat{q}\times\id_I)(g\hat{z},t)$ for all $(\hat{z}, t)\in
\hat{Z}\times I$ and $g\in G$. Then let $\hat{Z}_0=
(p\hat{f})^{-1}(0)$. Since $q_\br^n|\co \br^n\setminus\{ 0\}\to
\oc(X)\setminus\{ v\}$ is a covering map and $H_0$ is a $G$-map, it
follows that $H|\co(\hat{Z}\setminus\hat{Z}_0)\times
I\to\br^n\setminus\{ 0\}$ is a $G$-map. Finally, note that
$H(\hat{Z}_0\times I)= \{ 0\}$ and $\hat{Z}_0$ is $G$-invariant. Together
these observations imply that $H$ is a $G$-homotopy.

It follows that $H$ fits into a $G$-homotopy lifting problem:
$$
\begin{diagram}
\node{\hat{Z}} \arrow{s,l}{\times 0} \arrow{e,t}{\hat{f}} \node{M} \arrow{s,r}{p} \\
\node{\hat{Z}\times I}\arrow{ne,t,..}{\tilde H} \arrow{e,t}{H} \node{\br^n}
\end{diagram}
$$
By hypothesis, there is a $G$-homotopy $\tilde H\co \hat Z\times I
\to M$ such that $\tilde H_0= \hat f$ and $p \tilde H$ is
$\delta$-close to $H$. It follows that there is a unique map $\tilde
F\co Z\times I\to N$ making the following diagram commute:
$$\begin{CD}
\hat Z\times I @>{\tilde H}>> M \\
@V{\hat q\times\id_I}VV @VV{q_M}V \\
Z\times I @>{\tilde F}>> N
\end{CD}$$
It follows that $\tilde F$ is a  $\delta$-solution of the original
stratified problem. For this we use the assumption that
$q_{\br^n}\co\br^n\to\oc(X)$ is distance non-increasing together
with the following diagram
$$\begin{CD}
Z\times I @<{\hat{q}\times\id_I}<< \hat{Z}\times I @>{=}>>  \hat{Z}\times I @>{\hat{q}\times\id_I}>> Z\times I \\
@V{F}VV   @V{H}VV     @VV{p\tilde{H}}V @VV{(p/G)\circ\tilde{F}}V \\
\oc(X) @<{q_{\br^n}}<< \br^n @>{=}>> \br^n @>{q_{\br^n}}>> \oc(X)
\end{CD}$$
where the outer two squares are commutative and the middle square
commutes up to $\delta$.\qed
\end{proof}

We set up two basic lemmas on lifting homotopies across $q_{\br^n}:
\br^n \to \oc(X)$, which will be used in
Section~\ref{section:sucking}. Recall that $G$ is a finite subgroup
of $O(n)$.

\begin{lemma}\label{lem:G-homotopy}
Let $Z$ be a $G$-space, and denote the quotient map $p: Z \to Z/G$.
Suppose $f: Z \to \br^n$ is a $G$-map and $F: Z/G \times I \to
\oc(X)$ is a stratum-preserving homotopy such that $q_{\br^n}\circ f
= F(-,0) \circ p$. Then there is a unique $G$-homotopy $\tilde{F}: Z
\times I \to \br^n$ such that $\tilde{F}(-,0)=f$ and $q_{\br^n}
\circ \tilde{F} = F \circ (p \times \id_I)$.
\end{lemma}

\begin{proof}
Since $F \circ (p \times \id_I): Z \times I \to \oc(X)$ is a
stratum-preserving homotopy, by Lemma~\ref{lemma:basic quotient},
there exists a homotopy $\tilde{F}: Z \times I \to \br^n$ such that
$\tilde{F}(-,0)=f$ and $q_{\br^n} \circ \tilde{F} = F \circ (p
\times \id_I)$. Observe, since $\tilde{F}$ is stratum-preserving and
the restriction $q_{\br^n}|: \br^n \setminus \{0\} \to \oc(X)
\setminus \{v\}$ is a covering map, that $\tilde{F}$ is uniquely
determined.

Define $Z_0 := Z \setminus f^{-1}\{0\}$. Note $\tilde{F}$ restricts
to a homotopy $\hat{F}: Z_0 \times I \to \br^n \setminus \{0\}$. Let
$g \in G$ and $z \in Z_0$. Consider the paths
\[
\alpha := \hat{F}(g \cdot z,-),~ \beta := g \cdot
\hat{F}(z,-) : I \longrightarrow \br^n \setminus\{0\}.
\]
Note, since $f$ is a $G$-equivariant, that
\[
\alpha(0) = f(g \cdot z) = g \cdot f(z) = \beta(0).
\]
Note, since $p$ and $q_{\br^n}$ are $G$-invariant, that
\[
q_{\br^n} \circ \alpha = F(p(g\cdot z),-) = F(p(z),-) = q_{\br^n} \circ \hat{F}(z,-) = q_{\br^n}\circ \beta.
\]
Therefore, by the Path Lifting Property, we obtain that $\alpha =
\beta$. Thus $\hat{F}: Z_0 \times I \to \br^n\setminus\{0\}$ is a
$G$-homotopy. Hence $\tilde{F}: Z \times I \to \br^n$ is a
$G$-homotopy.\qed
\end{proof}

Recall that the open cone $\oc(X) = \br^n/G$ has the quotient
metric.

\begin{lemma}\label{lem:Lebesgue_control}
Let $Z$ be a topological space, and let $r>0$. There exists
$\epsilon_0>0$ such that: if $F: Z \times I \to \br^n \setminus
\mathrm{B}_r^n$ is a homotopy and $q_{\br^n} \circ F: Z \times I \to
\oc(X) \setminus \cone_r(X)$ is an $\epsilon_0$-homotopy, then $F$
is an $\epsilon_0$-homotopy.
\end{lemma}

\begin{proof}
Since $X$ is compact, there exists a finite cover $\cU$ by non-empty
open subsets $U \subseteq X$ such that the covering map $q: S^{n-1}
\to X$ evenly covers each $U \in \cU$.  Then the induced cover on
the metric subspace $X \times \{r\} \subset \oc(X)$ has a Lebesgue
number $\epsilon_0>0$. So the restriction of the induced cover
$\oc(\cU) := \{ \oc(U) ~|~ U \in \cU \}$ to the frustum $\oc(X)
\setminus \cone_r(X)$ has the same Lebesgue number $\epsilon_0 > 0$.

Let $z \in Z$. Then, since the track $q_{\br^n} F(\{z\}\times I)
\subset \oc(X) \setminus \cone_r(X)$ has diameter $< \epsilon_0$,
there exists $U \in \cU$ such that $q_{\br^n} F(\{z\}\times I)
\subset U \times (r,\infty)$. Let $V \subset q^{-1}(U)$ be the path
component of the point $F(z,0)$. Then, since $q|_V: V \to U$ is an
isometry, the track $F(\{z\}\times I) \subset V$ has diameter $<
\epsilon_0$. Thus $F$ is an $\epsilon_0$-homotopy.\qed
\end{proof}


\section{Piecing together bounded fibrations}\label{section:piecing}

\emph{Throughout Section~\ref{section:piecing}, we assume that $G$
is a finite subgroup of $O(n)$.}

Our goal here is to adapt \cite[Proposition~2.6]{Hughes1985Q} and
give a detailed proof. This result will be used in
Section~\ref{section:sucking}.

\begin{theorem}\label{thm:piecing}
Let $K \subset V \subset C \subset U$ be $G$-subsets of $\br^n$ such
that:
\begin{enumerate}
\item
$U,V$ are open subsets of $\br^n$,

\item
$C,K$ are closed subsets of $\br^n$, and

\item
$U$ (resp.~$V$) contains a metric neighborhood of $C$
(resp.~$K$).
\end{enumerate}
For every $\epsilon>0$ there exists $\delta>0$ such that for every
$\mu>0$ there exists $\nu>0$ satisfying: If $p: E \to \br^n$ is a
$G$-$\delta$-fibration over $U$ and a $G$-$\nu$-fibration over $V$,
then $p$ is a $G$-$(\epsilon,\mu)$-fibration over $(C,K)$ for the
class of compact, metric $G$-spaces.
\end{theorem}

The proof of the theorem is located at the end of this section. The
following corollary is required in Proposition~\ref{prop:equivariant
epsilon,nu} below.

\begin{corollary}\label{cor:piecing}
Let $0<r_1<r_2$ be given. For every $\epsilon
>0$ there exists $\delta >0$ such that for every $\mu > 0$ there
exists $\nu >0$ satisfying: if $p\co E\to\br^n$ is a
$G$-$\delta$-fibration over $\br^n$ and a $G$-$\nu$-fibration over
$\br^n\setminus\mathrm{B}^n_{r_1}$, then $p$ is a
$G$-$(\epsilon,\mu)$-fibration over $(\br^n,
\br^n\setminus\mathrm{B}_{r_2}^n)$ for the class of compact, metric
$G$-spaces.
\end{corollary}

\begin{proof}
This follows immediately from Theorem~\ref{thm:piecing} by taking
the $G$-subsets
\[
K = \br^n\setminus\mathrm{B}^n_{r_2} \subset  V = \br^n \setminus \mathrm{B}^n_{r_1} \subset C = U = \br^n.\quad\qed
\]
\end{proof}

\subsection{Homotopy extension}

First, we construct a certain strong deformation retraction. A
$G$-space $Z$ is \emph{normal} if $Z/G$ is normal and the quotient
map $Z \to Z/G$ is a closed map.

\begin{lemma}\label{lem:HomotopyExtension}
Let $B$ be an open $G$-subset of a normal $G$-space $Z$. Let $A
\subset B$ be a closed $G$-subset of $Z$. There is a $G$-homotopy
$R: (Z \times I) \times I \to Z \times I$ such that:
\begin{enumerate}
\item
$R(-,0)$ has image in $Z \times \{0\} \cup B \times I$, and

\item
$R(y,s)=y$ if $s=1$ or $y \in Z \times \{0\} \cup A \times I$.
\end{enumerate}
\end{lemma}

\begin{proof}
We may assume $Z\setminus B$ and $A$ are non-empty. Since $Z/G$ is a
normal space, by the Urysohn lemma, there exists a $G$-map
$\upsilon: Z \to I$ such that $\upsilon(Z\setminus B)=\{0\}$ and
$\upsilon(A)=\{1\}$. Define a $G$-homotopy
\[
R: (Z \times I) \times I \longrightarrow Z \times I ;\quad ((z,t),s) \longmapsto (z,st+(1-s)t\upsilon(z)).
\]
This function satisfies the required properties.\qed
\end{proof}

We adapt the Estimated Homotopy Extension Property
\cite[Prop.~2.1]{ChapmanFerry}.

\begin{lemma}\label{lem:EHEP}
Let $Y_0$ be a closed $G$-subset of a finite-dimensional, locally
compact, metric, separable $G$-space $Y$. For every $\lambda>0$: if
$h: Y \to \br^n$ is a $G$-map and $H_0: Y_0 \times I \to \br^n$ is a
$G$-$\lambda$-homotopy such that $h|Y_0 = H_0(-,0)$, then there
exists a $G$-$\lambda$-homotopy $H: Y \times I \to \br^n$ extending
$H_0$ such that $h = H(-,0)$.
\end{lemma}

\begin{proof}
We may assume $Y_0$ is non-empty. 
Note that 
there are finitely many fixed-point sets $(\br^n)^H$,
each of which is a vector subspace of $\br^n$, hence 
each $(\br^n)^H$ is an ANR.
Also note that 
$Y_1 := Y \times \{0\} \cup Y_0 \times I$ is a closed
$G$-subset of $Y \times I$.
Since $Y$ is a finite dimensional, locally compact, metric, separable
$G$-space, 
by a theorem of J.~Jaworowski
\cite[Thm.~2.2]{Jaworowski}, there exists an open $G$-neighborhood
$U_1$ of $Y_1$ in $Y$ and a $G$-map $H_U: U_1 \to \br^n$ extending
the $G$-map $h \cup H_0: Y_1 \to \br^n$.
By the tube lemma and
uniting open $G$-sets, there exists an open $G$-neighborhood $U_0'$
of $Y_0$ in $Y$ such that $U_0' \times I \subseteq U_1$. Since the
tracks of $H_U|Y_0 \times I=H_0$ have diameter $<\lambda$, by
continuity, there exists an open $G$-neighborhood $U_0 \subseteq
U_0'$ of $Y_0$ in $Y$ such that the tracks of $H_U|U_0 \times I$
have diameter $<\lambda$.

We may assume $Y \setminus U_0$ is non-empty. Since $Y/G$ is a
normal space, by the Urysohn lemma, there exists a $G$-map
$\upsilon: Y \to I$ such that $\upsilon(Y \setminus U_0) = \{0\}$
and $\upsilon(Y_0) = \{1\}$. Then, since $Y_1 \subseteq Y \times
\{0\} \cup U_0 \times I \subseteq U_1$, we can define a $G$-homotopy
\[
H: Y \times I \longrightarrow \br^n ;\quad (y,s) \longmapsto H_U(y,s \upsilon(y)).
\]
Note that $H$ extends $h \cup H_0$. Also note $\diam\, H(\{y\}
\times I) \leq \diam\, H_U(\{y\} \times I)< \lambda$ for all $y \in
Y$. This completes the proof.\qed
\end{proof}

We need an often-used corollary for close homotopies.

\begin{corollary}\label{cor:EHEP}
Let $X_0$ be a $G$-subset of a finite-dimensional, locally compact,
metric, separable $G$-space $X$. For every $\lambda>0$: if $f: X \to
\br^n$ is a $G$-map and $F: X \times I \to \br^n$ is a $G$-homotopy,
and if $F_0': X_0 \times I \to \br^n$ is a $G$-homotopy such that
$f|X_0 = F_0'(-,0)$ and $F_0'$ is $\lambda$-close to $F|X_0 \times
I$, then there exists a $G$-homotopy $F': X \times I \to \br^n$
extending $F_0'$ such that $f = F'(-,0)$ and $F'$ is $\lambda$-close
to $F$.
\end{corollary}

\begin{proof}
Define $Y_0 := X \times \{0\} \cup X_0 \times I$ and $Y := X \times
I$ and $h := F: Y \to \br^n$. Observe that the straight-line
homotopy
\[
H_0: Y_0 \times I \longrightarrow \br^n ;\quad ((x,t),s) \longmapsto
\begin{cases}
(1-s) F(x,0) + s f(x) & \text{if } t=0\\
(1-s) F(x,t) + s F_0'(x,t) & \text{if } x \in X_0
\end{cases}
\]
is a $G$-$\lambda$-homotopy such that $h|Y_0 = H_0(-,0)$. Then, by
Lemma~\ref{lem:EHEP}, there exists a $G$-$\lambda$-homotopy $H: Y
\times I \to \br^n$ extending $H_0$ such that $h = H(-,0)$. Define a
$G$-homotopy
\[
F' := H(-,1): X \times I \longrightarrow \br^n.
\]
Note $F'|X_0 \times I = F_0'$ and $F'(-,0) = H((-,0),1) = f$. Also
note $\| F'(y) - F(y) \| = \| H(y,1) - H(y,0) \| < \lambda$ for all
$y \in X \times I$. This completes the proof.\qed
\end{proof}

\subsection{Homotopy lifting}

We adapt the Stationary Lifting Property
\cite[Thm.~6.2]{Hughes1981}.

\begin{lemma}\label{lem:HomotopyLifting}
Let $\delta>0$. Let $A \subseteq \br^n$ be a $G$-subset. Let $Y$ be
a normal $G$-space. Suppose $p: E \to \br^n$ is a
$G$-$\delta$-fibration over $A$. If $H: Y \times I \to A$ is a
$G$-homotopy and $h: Y \to E$ is a $G$-map such that $ph = H(-,0)$,
then there exists a $G$-homotopy $\tilde{H}: Y \times I \to E$ such
that:
\begin{enumerate}
\item
$h = \tilde{H}(-,0)$,

\item
$p\tilde{H}$ is $\delta$-close to $H$, and

\item
$\tilde{H}(\{y\} \times I)=\tilde{H}(\{y\} \times \{0\})$ if
$H(\{y\} \times I)=H(\{y\}\times \{0\})$.
\end{enumerate}
\end{lemma}

\begin{proof}
Define a $G$-subset
\[
C := \{ y \in Y ~|~ H(\{y\} \times I) = H(\{y\} \times \{0\}) \}.
\]
Note that $C$ is the inverse image of $\{0\}$ under the $G$-map $(y
\mapsto \diam\,H(\{y\} \times I))$. Hence $C$ is a \emph{closed
$G$-$G_\delta$-subset of $Y$;} that is, $C$ is a closed $G$-subset
of $Y$ and $C$ is a countable intersection of open $G$-subsets of
$Y$. Then, since $Y/G$ is normal, by the strong Urysohn lemma, there
exists a $G$-map $\upsilon: Y \to I$ such that $C =
\upsilon^{-1}\{0\}$. Define a $G$-homotopy
\[
H^*: Y \times I \longrightarrow A ;\quad (y,s) \longmapsto
\begin{cases}
H(y,s/\upsilon(y)) & \text{if } 0 \leq s < \upsilon(y)\\
H(y,1) & \text{if } \upsilon(y) \leq s \leq 1.
\end{cases}
\]

Note that $\upsilon^{-1}\{0\} \subseteq C$ implies $H^*(-,0) =
H(-,0) = ph$. Since $p$ is a $G$-$\delta$-fibration over $A$, there
exists a $G$-homotopy $\tilde{H}^*: Y \times I \to E$ such that $h =
\tilde{H}^*(-,0)$ and $p \tilde{H}^*$ is $\delta$-close to $H^*$.
Now define $G$-homotopy
\[
\tilde{H}: Y \times I \longrightarrow E ;\quad (y,s) \longmapsto \tilde{H}^*(y,s \upsilon(y)).
\]
Note $\tilde{H}(-,0) = \tilde{H}^*(-,0) = h$ and $\| p
\tilde{H}(y,s) - H(y,s) \| = \| p \tilde{H}^*(y,s \upsilon(y)) -
H^*(y,s \upsilon(y)) \| < \delta$ for all $(y,s) \in Y \times I$.
Furthermore, if $y \in C \subseteq \upsilon^{-1}\{0\}$, then note
$\tilde{H}(y,s) = \tilde{H}^*(y,0) =  \tilde{H}(y,0)$. This
completes the proof.\qed
\end{proof}

\subsection{Blending bounds}

We adapt \cite[Lemma 4.7]{Hughes1981}. This result finds a jointly
close solution to the homotopy lifting problem for a prototypical
kind of homotopy. We say that a neighborhood $N$ of a subset $A$ of
a metric space $(X,d)$ is \emph{metric} if $N$ equals the
\emph{$\alpha$-neighborhood} $\{x \in X ~|~ d(x,A) < \alpha \}$ of
$A$ in $X$ for some $\alpha>0$.

\begin{lemma}\label{lem:BlendingBounds}
Let $K \subset \intr\,K' \subset V \subset C \subset U$ be
$G$-subsets of $\br^n$ such that:
\begin{enumerate}
\item
$U,V$ are open subsets of $\br^n$,

\item
$K', K$ are closed subsets of $\br^n$, and

\item
$U$ contains a metric neighborhood of $C$.
\end{enumerate}
Let $Z$ be a finite-dimensional, locally compact, metric, separable
$G$-space. For every $\epsilon>0$ there exists $\delta>0$ such that
for every $\mu>0$ there exists $\nu>0$ satisfying: if $p: E \to B$
is a $G$-$\delta$-fibration over $U$ and a $G$-$\nu$-fibration over
$V$, and if $F: Z \times I \to C$ is a $G$-homotopy with tracks in
$\{\intr\,K', C\setminus K \}$, and if $f: Z \to E$ is a $G$-map
such that $pf = F(-,0)$, then there exists a $G$-homotopy
$\tilde{F}: Z \times I \to E$ such that $f=\tilde{F}(-,0)$ and $p
\tilde{F}$ is $(\epsilon,\mu)$-close to $F$ with respect to $K$.
\end{lemma}

\begin{proof}
Let $\epsilon>0$. Select $0 < \delta < \epsilon$. Let $\mu>0$.
Select $0 < \nu \leq \min\{\epsilon-\delta,\mu\}$ such that the
$\nu$-neighborhood of $C$ is contained in $U$. The $G$-homotopy
$\tilde{F}: Z \times I \to E$ is constructed in two steps.

The first step is to consider the $G$-subsets
\begin{eqnarray*}
Z_1 &:=& \{ z \in Z ~|~ F(\{z\} \times I) \cap K \neq \varnothing \}\\
Z_1' &:=& \{ z \in Z ~|~ F(\{z\} \times I) \subset \intr\,K' \}.
\end{eqnarray*}
Since $K \subset \intr\,K'$, by hypothesis on $F$, we have $Z_1
\subseteq Z_1'$. Since $K$ and $K'$ are closed in $\br^n$, we have
that $Z_1$ is closed and $Z_1'$ is open in $Z$. Since $p$ is a
$\nu$-fibration over $V \supset K'$, there exists a $G$-homotopy
$\tilde{F}_0: Z_1' \times I \to E$ such that $f|Z_1' =
\tilde{F}_0(-,0)$ and $p \tilde{F}_0$ is $\nu$-close to $F|Z_1'
\times I$. Then, by Corollary~\ref{cor:EHEP}, there exists a
$G$-homotopy $F': Z \times I \to \br^n$ extending $p \tilde{F}_0$
such that $pf = F'(-,0)$ and $F'$ is $\nu$-close to $F$. Hence
$\tilde{F}_0$ is a partial lift of $F'$ and the image of $F'$ is
contained in $U$.

The second step is to use Lemma~\ref{lem:HomotopyExtension}.  Since
$Z$ is a normal $G$-space, there exists a $G$-homotopy $R: (Z \times
I) \times I \to Z \times I$ such that $R(-,0)$ has image in $Z
\times \{0\} \cup Z_1' \times I$ and that $R(y,s)=y$ if $s=1$ or $y
\in Z \times \{0\} \cup Z_1 \times I$. Consider the $G$-homotopy
$F'R: (Z \times I) \times I \to \br^n$ with initial $G$-lift $(f
\cup \tilde{F}_0)R(-,0): Z \times I \to E$. Since $p$ is a
$\delta$-fibration over $U$, by Lemma~\ref{lem:HomotopyLifting},
there exists a $G$-homotopy $\tilde{F}_I: (Z \times I) \times I \to
E$ such that:
\begin{itemize}
\item
$(f \cup \tilde{F}_0) R(-,0) = \tilde{F}_I(-,0)$,

\item
$p \tilde{F}_I$ is $\delta$-close to $F'R$, and

\item
$\tilde{F}_I(y,s) = (f \cup \tilde{F}_0)(y)$ if $y \in Z \times
\{0\} \cup Z_1 \times I$.
\end{itemize}

Now define a $G$-homotopy
\[
\tilde{F} := \tilde{F}_I(-,1): Z \times I \longrightarrow E.
\]
Note $\tilde{F}(-,0) = \tilde{F}_I((-,0),1) = f$. Also note $\| p
\tilde{F}(y) - F(y) \| \leq \| p \tilde{F}_I(y,1) - F'R(y,1) \| + \|
F'(y) - F(y) \| < \delta + \nu \leq \epsilon$ for all $y \in Z
\times I$. Furthermore, if $F(z,t) \in K$, then $z \in Z_1$, so note
$\| p \tilde{F}(y) - F(y) \| = \| p \tilde{F}_0(y) - F(y) \| < \nu
\leq \mu$. This completes the proof.\qed
\end{proof}

Finally, we adapt \cite[Theorem~4.8]{Hughes1981}. This result finds
a jointly close solution to the homotopy lifting problem for an
arbitrary homotopy.

\begin{proof}[Proof of Theorem~\ref{thm:piecing}]
To begin, we shall set up additional parameters. Select closed
$G$-subsets $K_1,K_2,K_3,C_1$ of $\br^n$ such that
\[
K \subset \intr\,K_1 \subset \intr\,K_2 \subset K_3 \subset V
\subset
C \subset \intr\,C_1 \subset C_1 \subset U
\]
and $U,K_1$ (resp.~$C\setminus K_1, \intr\,K_3$) contains a metric
neighborhood of $C_1,K$ (resp.~$C_1\setminus K, \intr\,K_2$). Let
$\epsilon>0$. Select $0 < \delta' \leq \epsilon/3$ such that $C_1
\setminus K$ (resp.~$\intr\,K_3$) contains the
$2\delta'$-neighborhood of $C \setminus K_1$ (resp.~$\intr\,K_2$).
Select $0 < \delta < \delta'$. Let $\mu>0$. Select $0 < \nu' \leq
\mu/3$. Select $0 < \nu \leq \min\{\delta'-\delta,\nu'\}$ such that
$U$ contains the $\nu$-neighborhood of $C_1$.

Next, let $Z$ be a compact, metric $G$-space. Let $F: Z \times I \to
C$ be a $G$-homotopy. Let $f: Z \to E$ be a $G$-map such that $pf =
F(-,0)$. Since $F$ is continuous, each $z \in Z$ has a neighborhood
$W^z$ in $Z$ and a finite partition $\cP^z$ of $I$:
\[
\cP^z = \{ 0 = t_0^z < \cdots < t_i^z < \cdots < t_{n^z}^z = 1 \}
\]
such that the partial-track $F(\{z\} \times [t_i^z,t_{i+1}^z])$ lies
in either $C\setminus K_1$ or $\intr\,K_2$. Since $Z$ is compact,
the open cover $\{ W^z | z \in Z\}$ admits a finite subcover, and
the common refinement $\cP$ of the associated partitions is finite:
\[
\cP = \{ 0 = t_0 < \cdots < t_i < \cdots < t_{n} = 1 \}.
\]
Thus, for each $z \in Z$ and $0 \leq i < n$, the partial-track
$F(\{z\} \times [t_i,t_{i+1}])$ lies in either $C\setminus K_1$ or
$\intr\,K_2$, depending on $z$.

Lastly, for each $0 \leq i \leq n$, we shall inductively define maps
$\tilde{F}_i: Z \times [0,t_i] \to E$ such that:
\begin{itemize}
\item
$\tilde{F}_0 = f$ and $\tilde{F}_i$ extends $\tilde{F}_{i-1}$ if
$i>0$,

\item
$p \tilde{F}_i$ is $(\epsilon,\mu)$-close to $F| Z \times
[0,t_i]$ with respect to $K$, and

\item
$p \tilde{F}_i| Z \times \{t_i\}$ is $(\delta',\nu')$-close to
$F| Z \times \{t_i\}$ with respect to $K$ if $i<n$.
\end{itemize}
Hence $\tilde{F} := \tilde{F}_n: Z \times I \to E$ shall be the
desired homotopy.

Since $pf = F(-,0)$, note $\tilde{F}_0 := f$ satisfies the above
properties. Assume, for some $0 \leq i < n$, that there exists
$\tilde{F}_i$ satisfying the three properties. Since $F$ is
continuous and $X$ is compact, by the tube lemma, there exists $t_i
< s_{i+1} < t_{i+1}$ such that $\diam\,F(\{z\} \times [t_i,s])<
\delta'$ (resp.~$< \nu'$) for all $z \in Z$ (resp.~if $F(z,t_i) \in
K$). Select $t_i < s_i < s_{i+1}$. Define a $G$-homotopy $H$ from
$p\tilde{F}_i|Z \times \{t_i\}$ to $F|Z \times \{t_{i+1}\}$ by
\[
H: Z \times [t_i,t_{i+1}] \longrightarrow \br^n ;\quad (z,t) \longmapsto
\begin{cases}
\displaystyle \frac{s_i-t}{s_i-t_i} p\tilde{F}_i(z,t_i) + \frac{t - t_i}{s_i-t_i} F(z,t_i) & \text{if } t \in [t_i,s_i]\\[2ex]
\displaystyle F\left(z, t_i + \frac{s_{i+1}-t_i}{s_{i+1}-s_i} (t-s_i) \right) & \text{if } t \in [s_i,s_{i+1}] \\[2ex]
\displaystyle F(z,t) & \text{if } t \in [s_{i+1},t_{i+1}].
\end{cases}
\]
Note, for all $(z,t) \in Z \times [t_i,t_{i+1}]$, that
\[
\| H(z,t) - F(z,t) \| \leq
\begin{cases}
\displaystyle \frac{s_i-t}{s_i-t_i} \| p\tilde{F}_i(z,t_i) - F(z,t_i) \| + \| F(z,t_i) - F(z,t) \| & \text{if } t \in [t_i,s_i]\\[2ex]
\displaystyle \diam\,F(\{z\} \times [t_i,s_{i+1}]) & \text{if } t \in [s_i,s_{i+1}]\\[2ex]
\displaystyle 0 & \text{if } t \in [s_{i+1},t_{i+1}].
\end{cases}
\]
Then observe that:
\begin{itemize}
\item
$H$ is $(2\delta',2\nu')$-close to $F|Z \times [t_i,t_{i+1}]$
with respect to $K$,

\item
$H$ has image in $C_1$, and

\item
$H$ has tracks in $\{\intr\,K_3, C_1 \setminus K\}$.
\end{itemize}
Since $p$ is a $\delta$-fibration over $U$ and a $\nu$-fibration
over $V$, by Lemma~\ref{lem:BlendingBounds} and the epsilonics in
its proof, there exists a $G$-homotopy $\tilde{H}: Z \times
[t_i,t_{i+1}] \to E$ such that $\tilde{F}_i(-,t_i) =
\tilde{H}(-,t_i)$ and $p \tilde{H}$ is $(\delta',\nu')$-close to $H$
with respect to $K$. Now define
\[
\tilde{F}_{i+1} := \tilde{F}_i \cup \tilde{H}: Z \times [0,t_{i+1}] \longrightarrow E.
\]
Note, for all $(z,t) \in Z \times [t_i,t_{i+1}]$, that
\[
\| p\tilde{F}_{i+1}(z,t) - F(z,t) \| \leq \| p\tilde{H}(z,t) - H(z,t) \| + \| H(z,t) - F(z,t) \|.
\]
Hence $p\tilde{F}_{i+1}$ is $(\epsilon,\mu)$-close to $F$ on $Z
\times [t_i,t_{i+1}]$ with respect to $K$. Furthermore
$p\tilde{F}_{i+1}|$ is $(\delta',\nu')$-close to $H|=F|$ on $Z
\times \{t_{i+1}\}$ with respect to $K$. This concludes the
inductive construction of the desired homotopy $\tilde{F}$.\qed
\end{proof}

\section{Equivariant sucking over Euclidean space}\label{section:sucking}

\emph{Throughout Section~\ref{section:sucking}, we assume that $G$
is a finite subgroup of $O(n)$, that $G$ acts freely on $S^{n-1}$
and on $M$, and that $M$ is a manifold of dimension $m > 4$.}

The following theorem is the main result herein; the proof is
located at the end of the section. It is an equivariant version of
the first part of Corollary~\ref{cor:euclidean_delta_sucking} and
appears in the Introduction as Theorem~\ref{intro thm: ortho
sucking}.

\begin{theorem}
\label{thm:equivariant_epsilon_delta_sucking} For every $\epsilon
> 0$ there exists $\delta > 0$ such that: if $p\co M\to\br^n$ is a
proper $G$-$\delta$-fibration, then $p$ is $G$-$\epsilon$-homotopic
to a $G$-MAF.
\end{theorem}

The proof of Theorem~\ref{thm:equivariant_epsilon_delta_sucking}
shows that $\delta$ is independent of $M$ but is dependent on
$\dim(M)$.

The following corollary is an equivariant version of the second part
of Corollary~\ref{cor:euclidean_delta_sucking}. The scaling trick in
the proof is due to Chapman \cite{Chapman} in the non-equivariant
case.

\begin{corollary}
\label{cor:equivariant_bounded_sucking}
If $p\co M\to \br^n$ is a proper $G$-bounded
fibration, then $p$ is $G$-boundedly homotopic to a $G$-MAF.
\end{corollary}

\begin{proof} Obtain $\delta>0$ from Theorem~\ref{thm:equivariant_epsilon_delta_sucking} with $\epsilon=1$.
Let $p\co M\to\br^n$ be a proper $G$-bounded fibration. That is,
there exists $\lambda > 0$ such that $p$ is a proper
$G$-$\lambda$-fibration. There exists $L>0$ such that $\lambda/L
<\delta$. Then, since $G\subseteq O(n)$, note that $\frac{1}{L}p$ is
a proper $G$-$\delta$-fibration. So, by
Theorem~\ref{thm:equivariant_epsilon_delta_sucking}, there exists a
$G$-$1$-homotopy $H: M\times I\to\br^n$ such that
$H(-,0)=\frac{1}{L}p$ and $p_1 := H(-,1)$ is a $G$-MAF. Therefore
the scaled map $L\cdot {H}\co M\times I \longrightarrow \br^n$ is a
bounded $G$-homotopy from $p$ to a $G$-MAF $L p_1$.\qed
\end{proof}

The rest of this section is devoted to the proof of Theorem
\ref{thm:equivariant_epsilon_delta_sucking}.

\begin{lemma}
\label{lem:cor_piecing} Let $0<r_1<r_2$ be given. For every
$\epsilon >0$ there exists $\delta >0$ satisfying: if $p\co
M\to\br^n$ is a proper $G$-$\delta$-fibration and a $G$-MAF over
$\br^n\setminus\mathrm{B}^n_{r_1}$, then, for every $\nu >0$, the
map $p$ is a $G$-$(\epsilon,\nu)$-fibration over $(\br^n,
\br^n\setminus\mathrm{B}_{r_2}^n)$ for the class of compact, metric
$G$-spaces.
\end{lemma}

\begin{proof}
This is an immediate consequence of Corollary~\ref{cor:piecing}.\qed
\end{proof}

\begin{lemma}
\label{lem:orbit_to_cover} Let $r>0$ and in the following
commutative diagram $$\begin{CD}
M @>{q_M}>> N \\
@V{p}VV  @VV{p/G}V \\
\br^n @>{q_{\br^n}}>> \oc(X)
\end{CD}$$
suppose that $p$ is a proper $G$-map and $p/G$ is a MAF over
$\oc(X)\setminus \cone_r(X)$. Then $p$ is a $G$-MAF over
$\br^n\setminus\mathrm{B}_r^n$.
\end{lemma}

\begin{proof}
Let $Z$ be a $G$-space, and denote the quotient map $q_Z: Z \to
Z/G$. Let $\epsilon > 0$. Let $f: Z \to M$ and $F: Z \times I \to
\br^n$ be the data for an $\epsilon$-lifting problem: $F(z,t) =
pf(z) \in \br^n \setminus \mathrm{B}_r^n$. Consider the induced
$\epsilon$-lifting problem, consisting of the continuous maps
$\bar{f}: Z/G \to N$ and $\bar{F}: Z/G \times I \to \oc(X)$ of
quotient spaces such that $\bar{F}(w,t) = q_M \bar{f}(w) \in \oc(X)
\setminus \cone_r(X)$. Since $p/G$ is an $\epsilon$-fibration over
$\oc(X) \setminus \cone_r(X)$, there exists an $\epsilon$-solution:
$\bar{H}: Z/G \times I \to N$ such that $\bar{H}(w,0)=\ol{f}(w)$ and
$d(p/G \circ \bar{H},\bar{F}) < \epsilon$. Define a $G$-invariant
map
\[
H := \bar{H} (q_Z \times \id_I): Z \times I \longrightarrow N.
\]
 Note that
$H(z,0)=\ol{f}q_Z(z)=q_M f$ and that $p/G \circ H$ is
$\epsilon$-close to $p/G \circ \bar{F} = q_{\br^n} F$.

Now, since $G$ acts freely on $M$ implies that $q_M$ is a covering
map, by the Homotopy Lifting Property, there exists a unique map
$\tilde{F}: Z \times I \to M$ such that $\tilde{F}(z,0)=f(z)$ and
$q_M \tilde{F} = H$. Note, for any $z \in Z$ and $\gamma \in G$,
since $f$ is $G$-equivariant and $H$ is $G$-invariant, that both the
paths $\gamma\tilde{F}(z,-)$ and $\tilde{F}(\gamma z,-)$ have common
initial point $\gamma f(z) = f(\gamma z)$ and have common
$q_M$-composition $H(z,-)=H(\gamma z,-)$.  Thus, by the uniqueness
property of path lifts in a covering space, we must have $\gamma
\tilde{F}(z,t) = \tilde{F}(\gamma z, t)$ for all $\gamma \in G, z
\in Z, t \in I$. Therefore $\tilde{F}: Z \times I \to M$ is a
$G$-homotopy. Since $q_{\br^n}$ is distance non-increasing,
note
\[
d(p\tilde{F},F) \leq d(q_{\br^n}p\tilde{F},q_{\br^n}F) = d(p/G \circ q_M \tilde{F},q_{\br^n}F) = d(p/G \circ H,q_{\br^n}F) < \epsilon.
\]
Thus $\tilde{F}$ is an $\epsilon$-solution to the lifting problem
given by $f$ and $F$. Therefore $p$ is an $\epsilon$-fibration over
$\br^n \setminus \ol{B}_r^n$, for all $\epsilon>0$.\qed
\end{proof}

The proof of the following lemma follows immediately from the definition.

\begin{lemma}
\label{lem:restriction} Let $r>0$ and  in the following commutative
diagram $$\begin{CD}
M @>{q_M}>> N @<{\incl}<< N'= g^{-1}(X\times (0,\infty))\\
@V{p}VV  @VV{p/G}V  @VV{p/G|}V\\
\br^n @>{q_{\br^n}}>> \oc(X) @<{\incl}<< X\times (0,\infty)
\end{CD}$$
suppose that $p$ is a proper $G$-map. If $p/G|$ is a MAF over
$X\times (r,\infty)$, then $p/G$ is a MAF over $\oc(X)\setminus
\cone_r(X)$. \qed\end{lemma}

\begin{lemma}
\label{lem:equivariant epsilon,MAF} For every $\epsilon
> 0$ there exists $\delta >0$ such that: if $p\co M\to\br^n$ is a
proper $G$-$\delta$-fibration, then $p$ is $G$-$\epsilon$-homotopic
to a map $p'$ that is a $G$-MAF over $\br^n\setminus\bar
B^n_\epsilon$.
\end{lemma}

\begin{proof}
Let $\epsilon > 0$ and $m > 4$ be given. By
Lemma~\ref{lem:Lebesgue_control}, there exists $0 < \epsilon_0 \leq
\epsilon/4$ such that: if $Z$ is a topological space and $F: Z
\times I \to \br^n\setminus \mathrm{B}_{\epsilon/2}^n$ is a homotopy
and $q_{\br^n} \circ F$ is an $\epsilon_0$-homotopy, then $F$ is an
$\epsilon_0$-homotopy. By Proposition~\ref{prop:HenriquesLeary}, the open cone $\oc(X) = \br^n/G$ with the induced metric has finite isometry type. 
Then, since $X = S^{n-1}/G$ is a closed smooth manifold, by
Proposition~\ref{prop:metric ANR} and Corollary~\ref{cor:chapman},
there exists $\delta
> 0$ such that: if $P$ is an $m$-manifold and $f: P \to X \times
(0,\infty)$ is a proper $\delta$-fibration over
$(\epsilon/2,\infty)$, then $f$ is $\epsilon_0$-homotopic rel
$f^{-1}(X\times(0,\epsilon - \epsilon_0])$ to a map $f'$ that is a
MAF over $X\times(\epsilon,\infty)$.

Let $p:M \to \br^n$ be a proper $G$-$\delta$-fibration. By
Proposition~\ref{prop:equivariant implies stratified}, the induced
map $g := p/G: N \to \oc(X)$ is a proper stratified
$\delta$-fibration. Consider the $m$-manifold $P := N \setminus
g^{-1}\{v\}$ and $f := g|P: P \to X \times (0,\infty)$. Since $f$ is
a proper $\delta$-fibration over $(\epsilon/2,\infty)$, there exists
an $\epsilon_0$-homotopy $H: P \times I \to X \times (0,\infty)$ rel
$f^{-1}(X \times (0,\epsilon - \epsilon_0])$ from $f$ to a MAF $f'$
over $X\times(\epsilon,\infty)$. This extends uniquely to an
$\epsilon_0$-homotopy $\bar{H}: N \times I \to \oc(X)$ rel
$g^{-1}(\cone_{\epsilon - \epsilon_0}(X))$ from $g$ to some $g' :=
\bar{H}(-,1)$. Note, by Lemma~\ref{lem:restriction}, that $g'$ is a
MAF over the frustum $\oc(X)\setminus \cone_\epsilon(X)$.

By Lemma~\ref{lem:G-homotopy}, there is a unique $G$-homotopy
$\tilde{H}: M \times I \to \br^n$ such that $\tilde{H}(-,0)=p$ and
$q_{\br^n} \circ \tilde{H} = \bar{H} \circ (q_M \times \id_I)$.
Observe that $\tilde{H}$ restricts to a constant homotopy on $M_0 :=
p^{-1}(\mathrm{B}_{\epsilon-\epsilon_0}^n)$. Furthermore, there
exists a restriction
\[
\bar{H}|: (N \setminus q_M(M_0)) \times I \longrightarrow \oc(X) \setminus c_{\epsilon-2\epsilon_0}(X) \subseteq \oc(X) \setminus c_{\epsilon/2}(X).
\]
Then the restriction $\tilde{H}|: (M \setminus M_0) \times I \to
\br^n \setminus \mathrm{B}_{\epsilon/2}^n$ exists and is an
$\epsilon_0$-homotopy. Hence $\tilde{H}: M \times I \to \br^n$ is an
$G$-$\epsilon$-homotopy. Finally, consider the map $p' :=
\tilde{H}(-,1): M \to \br^n$. Since $p'$ is a proper $G$-map
covering $g'$, by Lemma~\ref{lem:orbit_to_cover}, we conclude that
$p'$ is a $G$-MAF over $\br^n \setminus \mathrm{B}^n_\epsilon$.\qed
\end{proof}

\begin{proposition}
\label{prop:equivariant epsilon,nu} For every $\epsilon
> 0$ there exists $\delta >0$ satisfying: if $p\co M\to\br^n$ is a
proper $G$-$\delta$-fibration, then $p$ is $G$-$\epsilon$-homotopic
to a map $p'$ that is, for every $\nu > 0$, a proper
$G$-$(\epsilon,\nu)$-fibration over $(\br^n, \br^n\setminus
\mathrm{B}^n_\epsilon)$ for the class of compact, metric $G$-spaces.
\end{proposition}

\begin{proof}
Let $\epsilon >0$ be given. Let $r_1 = \epsilon/2$ and
$r_2=\epsilon$. Let $\delta_1 >0$ be given by
Lemma~\ref{lem:cor_piecing} for $\epsilon$, $r_1$, and $r_2$. Let
$\delta_2>0$ be given by Lemma~\ref{lem:equivariant epsilon,MAF} for
$\min\left(\epsilon/2, \delta_1/2\right)$.

Let $\delta=\min\left(\delta_2, \delta_1/2\right)$. If $p\co M\to
\br^n$ is a proper $G$-$\delta$-fibration, then
Lemma~\ref{lem:equivariant epsilon,MAF} implies that $p$ is
$G$-$\delta_1/2$-homotopic to a map $p'$, where $p'$ is a $G$-MAF
over $\br^n\setminus\mathrm{B}_{\epsilon/2}^n$. Now $p'$ is a
$G$-$\delta_1$-fibration. It follows from
Lemma~\ref{lem:cor_piecing} that, for every $\nu > 0$, the map $p'$
is a $G$-$(\epsilon, \nu)$-fibration over $(\br^n, \br^n\setminus
\mathrm{B}^n_\epsilon)$ for the class of compact, metric $G$-spaces.\qed
\end{proof}

\begin{definition}
Let $r > 0$. The \textbf{radial crush with parameter $r$} is the
$G$-map
\[
\rho_r\co \br^n \longrightarrow \br^n;\quad
x \longmapsto \begin{cases}
(1-\frac{r}{\norm{x}}) x & \text{if } r < \norm{x}\\
0 & \text{if } \norm{x} \leq r.
\end{cases}
\]
\end{definition}

\begin{lemma}
\label{lem:crush_prop} Let $r > 0$.
\begin{enumerate}
    \item There is a $G$-$r$-homotopy from $\id_{\br^n}$ to $\rho_r$.
    \item The map $\rho_r$ is distance non-increasing.
    \item The map $\rho_r$ is a $G$-MAF.
\end{enumerate}

\end{lemma}

\begin{proof}
Part (1) is given by the straight-line homotopy
\[
H: \br^n \times I \longrightarrow \br^n;\quad (x,t) \longmapsto
\begin{cases}
(1-\frac{tr}{\norm{x}}) x & \text{if } r < \norm{x}\\
(1-t)x & \text{if } \norm{x} \leq r.
\end{cases}
\]
This is easily checked to be a $G$-$r$-homotopy.

Part (2) follows from the calculation
\begin{eqnarray*}
\norm{\rho_r(x)-\rho_r(y)}^2
&=& \begin{cases}
0 & \text{if } \norm{x},\norm{y} \leq r\\
(\norm{y}-r)^2 & \text{if } \norm{x} \leq r \leq \norm{y}\\
\norm{x-y}^2 - 2r(1-\cos\theta)(\norm{x}+\norm{y}-r) & \text{if } r \leq \norm{x}, \norm{y}
\end{cases}\\
& \leq & \norm{x-y}^2~.
\end{eqnarray*}
The inequality in the second case follows from the Triangle
Inequality. The equality in the third case follows from the Law of
Cosines, where $\cos\theta = \frac{\gens{x}{y}}{\norm{x} \norm{y}}$.

For Part (3), consider the above straight-line homotopy $H = \{H_t:
\br^n \to \br^n\}_{t \in I}$. Observe, for each $t \in [0,1)$, that
$H_t$ is a $G$-homeomorphism. Moreover, for each $t \in [0,1)$,
note, for all $y \in \br^n$, that
\[
\norm{y - \rho_r H_t^{-1}(y)} = \begin{cases}
(1-t)r & \text{if } (1-t)r < \norm{y}\\
\norm{y} & \text{if } \norm{y} \leq (1-t)r
\end{cases}
\]
and hence, for all $x \in \br^n$, that
\[
\norm{\rho_r(x) - \rho_r H_t^{-1} \rho_r(x)} = \begin{cases}
(1-t)r & \text{if } r < \norm{x}\\
0 & \text{if } \norm{x} \leq r.
\end{cases}
\]
Let $\epsilon>0$. Select $T \in [0,1)$ such that $(1-T)r <
\epsilon$. Then, the straight-line homotopy from $\id_{\br^n}$ to
$\rho_r H_T^{-1}$ is a $G$-$\epsilon$-homotopy, and the
straight-line homotopy from $\id_{\br^n}$ to $H_T^{-1} \rho_r$ is a
$G$-$\epsilon$-homotopy, when measured in the target $\br^n$ using
$\rho_r$. Hence $\rho_r: \br^n \to \br^n$ is a
$G$-$\epsilon$-homotopy equivalence. Therefore, by $\text{(b)}
\Longrightarrow \text{(c)}$ in \cite[Theorem 3.4]{Prassidis}, we
conclude that $\rho_r$ is a $G$-approximate fibration.\qed
\end{proof}

\begin{lemma}
\label{lem:crush MAF} Let $a,\epsilon >0$. Suppose, for all $\nu >
0$, that $p\co M\to\br^n$ is a proper $G$-$(\epsilon,\nu)$-fibration
over $(\br^n, \br^n\setminus \mathrm{B}_a^n)$ for the class of
compact, metric $G$-spaces. If
$\rho=\rho_{a+\epsilon}\co\br^n\to\br^n$ is the radial crush with
parameter $a+\epsilon$, then the composite $\rho \circ p: M\to
\br^n$ is a $G$-MAF.
\end{lemma}

\begin{proof}
Let $Z$ be a compact, metric $G$-space. Let $\mu > 0$. Suppose the
following diagram is a $G$-homotopy lifting problem:
$$
\begin{diagram}
\node{Z} \arrow[2]{s,l}{\times 0} \arrow{e,t}{f} \node{M} \arrow{s,r}{p} \\
\node[2]{\br^n} \arrow{s,r}{\rho} \\
\node{Z\times I}\arrow{nne,t,..}{\tilde F} \arrow{e,t}{F} \node{\br^n}
\end{diagram}
$$
Then the following is also a $G$-homotopy lifting problem:
$$
\begin{diagram}
\node{Z} \arrow{s,l}{\times 0} \arrow{e,t}{pf} \node{\br^n} \arrow{s,r}{\rho} \\
\node{Z\times I}\arrow{ne,t,..}{\hat{F}} \arrow{e,t}{F} \node{\br^n}
\end{diagram}
$$
By Lemma~\ref{lem:crush_prop}(3) this latter problem has a
$G$-$\mu/2$-solution $\hat{F}\co Z\times I\to\br^n$. It follows that
$$
\begin{diagram}
\node{Z} \arrow{s,l}{\times 0} \arrow{e,t}{f} \node{M} \arrow{s,r}{p} \\
\node{Z\times I}\arrow{ne,t,..}{\tilde F} \arrow{e,t}{\hat{F}} \node{\br^n}
\end{diagram}
$$
is also a $G$-homotopy lifting problem, which by hypothesis has a
$G$-$(\epsilon, \mu/2)$-solution $\tilde{F}\co Z\times I\to M$ over
$(\br^n, \br^n\setminus \mathrm{B}_a^n)$. We show that $\rho p
\tilde{F}$ is $\mu$-close to $F$, as follows.

Let $(z,t)\in Z\times I$. There are two cases to consider. First
suppose  $\hat{F}(z,t)\in\br^n\setminus\mathrm{B}_a^n$. Then
$p\tilde{F}(z,t)$ is $\mu/2$-close to $\hat{F}(z,t)$ because
$\tilde{F}$ is a $(\epsilon, \mu/2)$-solution. Therefore, $\rho
p\tilde{F}(z,t)$ is $\mu/2$-close to $\rho\hat{F}(z,t)$ because
$\rho$ is distance non-increasing by Lemma~\ref{lem:crush_prop}(2).
Also $\rho\hat{F}(z,t)$ is $\mu/2$-close to ${F}(z,t)$ because
$\hat{F}$ is a $\mu/2$-solution. Thus, $\rho p\tilde{F}(z,t)$ is
${\mu}$-close to ${F}(z,t)$. Second suppose
$\hat{F}(z,t)\in\mathrm{B}_a^n$. In particular, $\rho \hat{F}(z,t) =
0$. Then $p\tilde{F}(z,t)$ is $\epsilon$-close to $\hat{F}(z,t)$
because $\tilde{F}$ is an $\epsilon$-solution. Therefore,
$p\tilde{F}(z,t)\in\mathrm{B}_{a+\epsilon}^n$ and so $\rho
p\tilde{F}(z,t)=0$. Also $F(z,t)$ is  $\mu/2$-close to
$\rho\hat{F}(z,t)=0$ because $\hat{F}$ is a $\mu/2$-solution. Thus,
$F(z,t)$ is $\mu/2$-close to $\rho p\tilde{F}(z,t)$. Hence $\rho
p\tilde{F}$ is $\mu$-close to $F$. In any case, we have shown 
$\rho p \tilde{F}$ is $\mu$-close to $F$.

Thus $\rho p: M \to \br^n$ is an approximate $G$-fibration for the
class of compact, metric $G$-spaces. Therefore, by a result of
S.~Prassidis \cite[Prop.~2.18]{Prassidis}, we conclude that $\rho p$
is an approximate $G$-fibration for the class of all $G$-spaces.\qed
\end{proof}

\begin{proof}[Proof of Theorem \ref{thm:equivariant_epsilon_delta_sucking}]
Given $\epsilon > 0$, let $\delta$ be given by
Proposition~\ref{prop:equivariant epsilon,nu} for $\epsilon/3$. It
follows that $p$ is $G$-$\epsilon/3$-homotopic to a map $p'\co
M\to\br^n$ that is, for all $\nu
>0$, a proper $G$-$(\epsilon/3,\nu)$-fibration
over $(\br^n, \br^n\setminus\mathrm{B}_{\epsilon/3}^n)$ for the
class of compact, metric $G$-spaces. Let
$\rho=\rho_{2\epsilon/3}\co\br^n\to\br^n$, the radial crush with
parameter $2\epsilon/3$. Lemma~\ref{lem:crush MAF} implies that
$\rho p'\co M\to\br^n$ is a $G$-MAF. Lemma~\ref{lem:crush_prop}(1)
implies that there is a $G$-$2\epsilon/3$-homotopy from
$\id_{\br^n}$ to $\rho$. Therefore, $p$ is $G$-$\epsilon$-homotopic
to $\rho p'$, a $G$-MAF.\qed
\end{proof}

\section{Bounded fibrations from discontinuous actions}\label{sec:bounded_fibrations}

First, we show certain $\Gamma$-spaces $W$ admit an equivariant map
to the real line.

\begin{proposition}\label{prop:existGmap}
Let $W$ be a connected topological manifold equipped with a free,
discontinuous $C_\infty$-action (resp. with a free, discontinuous
$D_\infty$-action). Then there exists a $C_\infty$-map (resp.
$D_\infty$-map) $\ol{f}: W \to \br$. Hence $\overline{f}$ is the
infinite cyclic cover of a map (resp. $C_2$-map) $f: M \to S^1$ of
topological manifolds.
\end{proposition}

\begin{proof}
Consider the infinite group $\Gamma := C_\infty$ (resp. $\Gamma :=
D_\infty$). Since the $\Gamma$-action on $W$ is free and
discontinuous, the quotient map $q: W \to N := W/\Gamma$ is a
regular covering map. Since $N$ is a connected topological manifold,
by \cite[Essay~III, Theorem~4.1.3]{KirbySiebenmann}, there exists a
(locally finite) simplicial complex $N'$ and a homotopy equivalence
$h: N' \to N$. Then the pullback $W' := h^*(W)$ has a canonical map
$q': W' \to N'$. Since $q$ is a covering map, we have that $q'$ is a
covering map. Then $W'$ has an induced simplicial structure and
free, simplicial $\Gamma$-action. By replacing the simplicial
structure on $W$ with the first barycentric subdivision, we obtain
that $W'$ is a free $\Gamma$-CW-complex. Furthermore, by covering
space theory, the canonical map $g: W' \to W$ (satisfying $q \circ g
= h \circ q'$) is a $\Gamma$-homotopy equivalence. Then there exists
a $\Gamma$-map $g': W \to W'$ that is a $\Gamma$-homotopy inverse to
$g$. Since $W'$ is a $\Gamma$-CW-complex with isotropy in the family
$fin$ of finite subgroups of $\Gamma$, by
\cite[Theorem~I.6.6]{tomDieck}, there exists a $\Gamma$-map $c: W'
\to E_{fin}\Gamma = \br$. Now define the desired $\Gamma$-map by
$\overline{f} := c \circ g': W \longrightarrow \br$.\qed
\end{proof}

The main theorem herein shows that $C_2$-bands unwrap to
$C_2$-bounded fibrations.

\begin{definition}
A \textbf{manifold band} $(M,f)$ consists of a closed, connected
topological manifold $M$ and a continuous map $f: M \to S^1$ such
that the induced infinite cyclic cover $\ol{M} := f^*(\br)$ is
connected and finitely dominated (cf.~\cite[Defn.
15.3]{HughesRanicki}). Let the circle $S^1$ have the $C_2$-action
generated by complex conjugation $R: z \mapsto \ol{z}$. A
\textbf{$C_2$-manifold band} $(M,f,R)$ consists of a manifold band
$(M,f)$ and a $C_2$-action on $M$, generated by a homeomorphism $R:
M \to M$, such that the continuous map $f: M \to S^1$ is
$C_2$-equivariant.
\end{definition}

We do not assume that the $C_2$-action on $M$ is free.

\begin{definition}
Let $(M,f,R)$ be a $C_2$-manifold band. The \textbf{infinite cyclic
cover} $(\ol{M},\ol{f},R)$ is defined by the pullback diagram
\[\begin{CD}
\ol{M} @>{\ol{f}}>>  & \br\\
@V{p}VV & @V{\exp}VV\\
M @>{f}>> & S^1
\end{CD}
\]
and the diagonal $C_2$-action
\[
R: \ol{M} \longrightarrow \ol{M}; \quad (x,t) \longmapsto (Rx,-t).
\]
The covering translation is defined by
\[
T: \ol{M} \longrightarrow \ol{M}; \quad (x,t) \longmapsto (x,t+1)
\]
and satisfies the dihedral relation $RT = T^{-1} R$. Note that all
the maps in the diagram are proper and $C_2$-equivariant, and
moreover that $\ol{f}$ is $D_\infty$-equivariant. Observe that
$(M,f)$ is a connected manifold band if and only if $\ol{M}$ is a
connected, finitely dominated manifold.
\end{definition}

\begin{example}
Let $n > 0$. An example of a free $C_2$-manifold band is
\[
(S^n\times S^1, \proj_{S^1}, (x,z) \longmapsto (-x, \ol{z})) ~.
\]
The following infinite cyclic cover possesses a free
$D_\infty$-action:
\[
(S^n \times \br, \proj_{\br}, (x,t) \longmapsto (-x,-t) ) ~.
\]
\end{example}

The following statements are equivariant versions of \cite[Props.
17.13, 17.14]{HughesRanicki}.

\begin{lemma}\label{Lem_SlidingDomination}
Let $(M,f,R)$ be a $C_2$-manifold band.  Then there exists a
$C_2$-homotopy
\[
\{K_s: \ol{M} \times \br \to \ol{M} \times \br\}_{s \in I},
\]
called a \textbf{$C_2$-sliding domination}, and a constant $N > 1$
such that:
\begin{enumerate}
\item
$K_0 = \id_{\ol{M} \times \br}$,

\item
$\proj_{\br} \circ K_s = \proj_{\br}$ for all $s \in I$,

\item
$K_s| \Graph(\ol{f}: \ol{M} \to \br)$ is the inclusion for all
$s \in I$, and

\item
$K_1(\ol{M} \times \br) \subseteq \{(x,t) ~|~ | \ol{f}(x) -t |
\leq N\}$.

\end{enumerate}
\end{lemma}

We perform a folding trick to construct $K$ indirectly.

\begin{proof}
By \cite[Proposition 17.13]{HughesRanicki}, there exists a homotopy
\[
\{K'_s: \ol{M} \times \br \to \ol{M} \times \br\}_{s \in I},
\]
called a \textbf{sliding domination}, and a constant $N>1$
satisfying Properties (1)--(4).  Then, by Property (3) for $K'$,
there is a well-defined, continuous $C_2$-map
\[
K: I \times \ol{M} \times \br \longrightarrow \ol{M} \times \br; \quad (s,x,t) \longmapsto
\begin{cases}K'_s(x,t) & \text{if } \ol{f}(x) \leq t \\ R K'_s(x,t) & \text{if } \ol{f}(x) \geq t. \end{cases}
\]
One checks easily that $K$ and $N$ satisfy Properties~(1)--(4).\qed
\end{proof}

\begin{theorem}\label{thm:unwrapping}
Let $(M,f,R)$ be a $C_2$-manifold band.  Then the proper map
$\ol{f}: \ol{M} \to \br$ is a $C_2$-bounded fibration.
\end{theorem}

\begin{proof}
We show that $\ol{f}$ is an $(C_2,N+\epsilon)$-fibration for any
$\epsilon > 0$.  Consider the $C_2$-sliding domination $K$ with
constant $N>1$ of Lemma \ref{Lem_SlidingDomination}. Let $g: Z \to
\ol{M}$ be a $C_2$-map and $G: Z \times I \to \br$ be a
$C_2$-homotopy such that $G(z,0) = \ol{f} g(z)$.  First, define a
$C_2$-homotopy
\[
\hat{G}: Z \times I \longrightarrow \ol{M} \times \br; \quad (z,t) \longmapsto
(g(z), G(z,t)).
\]
Next, define a $C_2$-homotopy
\[
\tilde{G}:  Z \times I \longrightarrow \ol{M}; \quad \tilde{G} :=
\proj_{\ol{M}} \circ K_1 \circ \hat{G}.
\]
Note Lemma \ref{Lem_SlidingDomination}(2) implies $K_1\hat{G} =
(\tilde{G}, G)$. Then Lemma \ref{Lem_SlidingDomination}(4) implies
$d(\ol{f} \tilde{G}, G) \leq N$. Therefore $\ol{f}: \ol{M} \to \br$
is a $C_2$-bounded fibration.\qed
\end{proof}

\section{Dihedral wrapping up over the real line}\label{sec:wrappingup}

The main theorem herein constructs a dihedral covering translation.
The non-$C_2$-version is given in \cite[Theorem
17.1]{HughesRanicki}, and we modify some of its techniques.

\begin{theorem}\label{thm:WrappingUp}
Let $W$ is a topological manifold of dimension $>4$ and equipped
with a free $C_2$-action generated by $R\co W\to W$. Suppose $p:
W\to\br$ is a $C_2$-MAF.
\begin{enumerate}
\item
There exists a cocompact, free, discontinuous $D_\infty$-action
on $W$ extending the $C_2$-action.

\item
The $C_2$-MAF $p: W\to\br$ is properly $C_2$-homotopic to a
$D_\infty$-MAF $\tilde{p}: W \to \br$.
\end{enumerate}
\end{theorem}

We shall construct a covering translation $J_1: W\to W$ satisfying
the dihedral relation $RJ_1 = J_1^{-1}R$.  The basic technique is to
cut, fill in, and paste embeddings.

\begin{proof}[Proof of Theorem \ref{thm:WrappingUp}(1)]
First, we decompose $W$ and obtain a certain isotopy $G$. Consider a
decomposition $W = C_- \cup_{\partial_-B} B \cup_{\partial_+B} C_+$
into closed subsets:
\begin{gather*}
B := p^{-1}[-\mu,\mu], \quad \partial_- B := p^{-1}\{-\mu\}, \quad \partial_+ B := p^{-1}\{\mu\},\\
C_- := p^{-1}(-\infty,-\mu],\quad C_+ := p^{-1}[\mu,\infty).
\end{gather*}
Also consider the auxiliary compact subsets
\[
K := p^{-1}[-1,-\mu], \quad \partial^- K := p^{-1}\{-1\}.
\]
Define an isotopy $g: I \times \br \to \br$ from the identity:
\[
\left\{ g_s: \br \longrightarrow \br ;\quad
t \longmapsto s+t \right\}_{s \in I}.
\]
Select $0 < \mu < \frac{3}{7}$. Then, by Corollary~\ref{cor:AIC},
there exists an isotopy $G: I \times W \to W$ of homeomorphisms such
that $G_0=\id_W$ and $p G_s$ is $\mu/3$-close to $g_s p$ for all $s
\in I$.

Second, we construct the homeomorphism $J_1: W \to W$ in four steps,
as follows.  In fact, we shall construct an isotopy $J: I \times W
\to W$ of homeomorphisms from $J_0=\id_W$ to the desired $J_1$ such
that $pJ_s$ is $\mu$-close to $g_s p$.

\begin{enumerate}
\item Define an isotopy $H$ of embeddings from the inclusion:
\[
\left\{ H_s := RG_s^{-1}R|K \cup G_s|C_+: K \cup C_+ \longrightarrow W \right\}_{s \in I}.
\]
Indeed each $H_s$ is injective, since $RG_s^{-1}R(K) \cap
G_s(C_+) = \varnothing$.

\item
Since $K \cup \partial_+ B$ is compact and $H|K \cup \partial_+
B$ is the restriction of a proper isotopy of a neighborhood of
$K \cup
\partial_+ B$ in $W$, by the Isotopy Extension Theorem of Edwards--Kirby
\cite[Cor.~1.2]{EdwardsKirby}, we obtain that $H|K \cup
\partial_+ B$ extends to an isotopy $\bar{H}$ from the identity:
\[
\left\{ \bar{H}_s: W \longrightarrow W \right\}_{s \in I}.
\]
Then $H| K \cup C_+$ extends to an isotopy of embeddings from
the inclusion:
\[
\left\{ J^+_s := H_s|K \cup \bar{H}_s| B \cup H_s|C_+ : K \cup B \cup C_+ \longrightarrow W \right\}_{s \in I}.
\]
Since each $\bar{H}_s(B)$ is the unique compact subset of $W$
with frontier $H_s(\partial_- B \cup \partial_+ B)$, we obtain
$\bar{H}_s(\intr\,B) \cap H_s(K \cup C_+) = \varnothing$. So
$J^+_s$ is indeed injective.

\item
Since $H_s(\partial_-K) \subset B$, we obtain $C_+ \subset
J^+_s(K \cup B \cup C_+)$. Then we can define an isotopy $J^-$
of embeddings from the inclusion:
\[
\left\{ J^-_s := R (J^+_s)^{-1}R|C_-: C_- \longrightarrow W \right\}_{s \in I}.
\]

\item
Note $J^-_s|K = R G_s^{-1}R|K = J^+_s|K$. Therefore, we can
define an isotopy $J$ of homeomorphisms from $\id_W$ as a union:
\[
\left\{ J_s := J^-_s|C_- \cup J^+_s|(B \cup C_+): W \longrightarrow W \right\}_{s \in I}.
\]
\end{enumerate}

Third, we verify that $pJ_s$ is $\mu$-close to $g_s p$ for all $s
\in I$. Note, for all $x \in C_-$, writing $y := (J^+_s)^{-1}(Rx)$,
that $|p J^-_s(x)-g_sp(x)| = |pR(J^+_s)^{-1}R(x) - R g_s^{-1} R
p(x)| = |g_s p(y) - p J^+_s(y)|$. Hence it suffices to verify that
$pJ^+_s$ is $\mu$-close to $g_s p$. If $x \in C_+$ then $|pJ^+_s(x)
- g_sp(x)| = |pG_s(x)-g_sp(x)| < \mu/3$. If $x \in K$, writing $y :=
G_s^{-1}(Rx)$, then note $|pJ^+_s(x) - g_sp(x)| = |pRG_s^{-1}R(x) -
Rg_s^{-1}Rp(x)| = |g_s p(y) - pG_s(y)| < \mu/3$. Otherwise, suppose
$x \in B$. Since $\bar{H}$ is an isotopy, we have $s-2\mu/3 < \inf
pRG_s^{-1}R(\partial_-B) \leq p \bar{H}_s(x) \leq \sup
pG_s(\partial_+ B) < s+2\mu/3$. Then $|pJ^+_s(x) - g_sp(x)| =
|p\bar{H}_s(x)-g_sp(x)| \leq |g_s^{-1}p\bar{H}_s(x)| + |p(x)| <
2\mu/3 + \mu/3=\mu$. Therefore $pJ_s$ is $\mu$-close to $g_s p$ for
all $s \in I$.

Fourth, we verify that $J_1$ satisfies the dihedral relation. On the
one hand, recall $C_+ \subset J^+_s(K \cup B \cup C_+)$ and $J^-_s|K
= J^+_s|K$. On the other hand, observe $\mu/3<\frac{1}{7}$ implies
$J^+_1(\partial_- B) \subset \intr\,C_+$ hence $C_- \cup B \subset
J_1(C_-)$. Then
\begin{gather*}
J_1^{-1}R|C_- = (J^+_1)^{-1}R|C_- = R J^-_1 |C_- = RJ_1|C_-\\
J_1^{-1}R|(B \cup C_+) = (J^-_1)^{-1}R|(B \cup C_+) = R J^+_1|(B
\cup C_+) = RJ_1|(B \cup C_+).
\end{gather*}
Therefore the dihedral relation $J_1^{-1}R=RJ_1$ holds. (In fact,
the relation $J_s^{-1}R=RJ_s$ holds for all $s \in [7\mu/3, 1]$.)
Hence $\{R,J_1\}$ generates a $D_\infty$-action on $W$.

Finally, we verify that the $D_\infty$-action is cocompact. Let $x
\in W$. Define a neighborhood $U := p^{-1}(p(x)-\mu,p(x)+\mu)$ of
$x$ in $W$. Since $pJ_1$ is $\mu$-close to $g_1p$ and
$\mu<\frac{1}{2}$, we have $pJ_1(U) \cap p(U) = \varnothing$. Thus
the $C_\infty$-action generated by $J_1$ is free and discontinuous.
That is, the quotient map $W \to W/J_1$ is a covering map.
Furthermore, define a $J_1 R$-invariant subset
\[
V := p^{-1}[0,\infty) \cap J_1 p^{-1}(-\infty,0]
\]
with frontier $\partial V := p^{-1}\{0\} \sqcup J_1 p^{-1}\{0\}$ in
$W$. Since $pJ_1$ is $\mu$-close to $g_1 p$, it follows that $V$ is
compact. Note $W  = \bigcup_{n \in \bz} (J_1)^n(V)$ and $(J_1)^n(V)
\cap (J_1)^m(V) \subset \partial V$ for all $m,n \in \bz$. Therefore
$V$ is a fundamental domain for $J_1$ and $W/J_1$ is compact.\qed
\end{proof}

We shall construct a $D_\infty$-MAF $\tilde{p}: W \to \br$ according
to the following outline: (1) use Urysohn functions to interpolate
crudely between $p$ and $g_1 p J_1^{-1}$, (2) use relative sucking
to sharpen the bounded fibration to an approximate fibration, (3)
use a crushing map to force $C_2$-equivariance about $\frac{1}{2}$,
and (4) use the fundamental domain $V$ to copy-and-paste across $W$
to force $C_\infty$-equivariance.

\begin{proof}[Proof of Theorem \ref{thm:WrappingUp}(2)]
It suffices to construct the desired $D_\infty$-MAF $\tilde{p}: W
\to \br$, since the straight-line homotopy is a proper
$C_2$-homotopy from $p$ to $\tilde{p}$:
\[
H: I \times W \longrightarrow \br; \quad (s,x) \longmapsto (1-s) p(x) + s \tilde{p}(x).
\]

First, by the relative sucking principle
(Prop.~\ref{prop:rel_sucking}), we may re-select $0 < \mu <
\frac{1}{48}$ so that: if $f: W \to \br$ is a proper
$2\mu$-fibration which is an approximate fibration over
$(-\infty,\frac{7}{24}) \cup (\frac{17}{24},\infty)$, then $f$ is
homotopic rel $f^{-1}((-\infty,\frac{1}{4}) \cup
(\frac{3}{4},\infty))$ to an approximate fibration $f': W \to \br$.
Then, since $\mu<\frac{1}{3}$, by the Urysohn lemma, there exists a
map $u: W \to I$ such that $p^{-1}(-\infty,\frac{1}{3}] \subseteq
u^{-1}\{0\}$ and $J_1 p^{-1}[-\frac{1}{3},\infty) \subseteq
u^{-1}\{1\}$. Define a proper map
\[
p': W \longrightarrow \br;\quad x \longmapsto u(x) + p J_{u(x)}^{-1}(x).
\]

Second, note $|p'(x) - p(x)| = |g_s p(y) - p J_s(y)| < \mu$ for all
$x \in W$, where we abbreviate $s := u(x)$ and $y := J_s^{-1}(x)$.
Since $p$ is a $\mu$-fibration, we conclude that $p'$ is a
$2\mu$-fibration. Furthermore, since $\mu < \frac{1}{24}$ implies
$(p')^{-1}(-\infty,\frac{7}{24}) \subset
p^{-1}(-\infty,\frac{1}{3}]$, we obtain that $p'$ restricts to the
approximate fibration $p$ over $(-\infty,\frac{7}{24})$. Moreover,
since $\mu < \frac{1}{48}$ implies $(p')^{-1}(\frac{17}{24},\infty)
\subset J_1 p^{-1}[-\frac{1}{3},\infty)$, we obtain that $p'$
restricts to the approximate fibration $g_1 p J_1^{-1}$ over
$(\frac{17}{24},\infty)$. Therefore $p'$ is homotopic rel
$(p')^{-1}((-\infty,\frac{1}{4}) \cup (\frac{3}{4},\infty))$ to an
approximate fibration $p'': W \to \br$.

Third, consider a proper cell-like map
\[
\kappa: \br \longrightarrow \br; \quad
t \longmapsto \begin{cases}
2t & \text{if } t \in [0,\frac{1}{4}]\\
\frac{1}{2} & \text{if } t \in [\frac{1}{4},\frac{3}{4}]\\
2t+1 & \text{if } t \in [\frac{3}{4},1]\\
t & \text{if } t \in (-\infty,0] \cup [1,\infty).
\end{cases}
\]
By the composition principle \cite[p.~38]{Coram}, we obtain that
$p''' := \kappa \circ p'': W \to \br$ is also an approximate
fibration. Furthermore, since $p R = R p$ and $J_1 R = R J_1^{-1}$,
note $p'''\circ J_1R = g_1 R \circ p'''$; roughly speaking, $p'''$
is $C_2$-equivariant about $\frac{1}{2}$.

Fourth, recall that $V = p^{-1}[0,\infty) \cap J_1 p^{-1}(-\infty,0]$.
It follows that $p'''|p^{-1}\{0\}=p|$ and $p'''|J_1p^{-1}\{0\} = g_1 p
J_1^{-1}|$. Then we can define a proper map
\[
\tilde{p}: W \longrightarrow \br; \quad x \longmapsto g_1^m p''' J_1^{-m}(x) \text{ where } x \in J_1^m(V).
\]
Since $\tilde{p} \circ J_1 = g_1 \circ \tilde{p}$ and $\tilde{p}
\circ J_1 R = g_1 R \circ \tilde{p}$, we obtain that $\tilde{p}$ is
a proper $D_\infty$-map. Furthermore, by the uniformization
principle \cite[p.~43]{Coram}, we obtain that $\tilde{p}$ is an
approximate fibration. Therefore, by Lemma~\ref{lem:SharpenMAF} (see
below), we conclude that $\tilde{p}$ is in fact a $D_\infty$-MAF.\qed
\end{proof}

\begin{lemma}\label{lem:SharpenMAF}
Let $W$ be a topological manifold equipped with a free,
discontinuous $D_\infty$-action. Suppose $\tilde{p}: W \to \br$ is a
proper $D_\infty$-map and an approximate fibration. Then $\tilde{p}$
is a $D_\infty$-MAF.
\end{lemma}

\begin{proof}
First, we show that the induced $C_2$-map $\tilde{p}/J_1: W/J_1 \to
\br/\bz=S^1$ is indeed a $C_2$-MAF. By Coram--Duvall's
uniformization principle \cite[p.~43]{Coram}, we conclude that
$\tilde{p}/J_1: W/J_1 \to S^1$ is an approximate fibration. Since
$C_2$ acts freely on $W$, that the fixed-set restriction
$(\tilde{p}/J_1)^R: (W/J_1)^R = \varnothing \to (S^1)^R = \{1,-1\}$
is trivially an approximate fibration. By Jaworowski's recognition
principle \cite[Thm.~2.1]{Jaworowski}, both the finite-dimensional
spaces $W/J_1$ and $S^1$ are $C_2$-ENRs, hence they are $C_2$-ANRs
for the class of separable metric spaces. Therefore, by Prassidis's
recognition principle \cite[Thm.~3.1]{Prassidis}, we conclude that
$\tilde{p}/J_1$ is a $C_2$-approximate fibration.

Now, we show, by an elementary argument, that $\tilde{p}: W \to \br$
is a $D_\infty$-MAF. Let $Z$ be a $D_\infty$-space, and let
$0<\epsilon < \frac{1}{2}$. Let $F: Z \times I \to \br$ be a
$D_\infty$-homotopy. Let $f: Z \to W$ be a $D_\infty$-map such that
$\tilde{p} f = F(-,0)$. Since $\tilde{p}/J_1: W/J_1 \to S^1$ is a
$C_2$-$\epsilon$-fibration, there exists a $C_2$-homotopy
$\widetilde{F/J_1}: Z/J_1 \times I \to W/J_1$ such that $f/J_1 =
(\widetilde{F/J_1})(-,0)$ and $(\tilde{p}/J_1)(\widetilde{F/J_1})$
is $\epsilon$-close to $F/J_1$. Furthermore, consider the
normalization map
\[
\Theta: \br^2 \setminus \{0\} \longrightarrow S^1;
\quad y \longmapsto y/\|y\|.
\]
Since $\epsilon < \frac{1}{2}$, there is the unique great-circle
$C_2$-$\epsilon$-homotopy
\begin{gather*}
H: I \times (Z/J_1 \times I) \longrightarrow S^1 \subset \br^2 ;\\
(s,x) \longmapsto \Theta\left(
(1-s)(\tilde{p}/J_1)(\widetilde{F/J_1})(x) + s (F/J_1)(x) \right).
\end{gather*}
Since $W \to W/D_\infty$ is a covering map, there exists a unique
$D_\infty$-homotopy $\tilde{F}: Z \times I \to W$ such that $f =
\tilde{F}(-,0)$ and $\tilde{F}/J_1 = \widetilde{F/J_1}$. Then, since
$\br \to S^1$ is a covering map, there exists a unique
$C_\infty$-homotopy $\tilde{H}: I \times (Z \times I) \to \br$ such
that $\tilde{p}\tilde{F} = \tilde{H}(0,-)$ and $\tilde{p}f =
\tilde{H}(-,(-,0))$ and $F = \tilde{H}(1,-)$ and $\tilde{H}/J_1 =
H$. Furthermore, since $H$ is an $\epsilon$-homotopy and the
quotient map $\br \to S^1$ is a local isometry, we conclude that
$\tilde{H}$ is an $\epsilon$-homotopy. Hence $\tilde{p} \tilde{F}$
is $\epsilon$-close to $F$. Thus $\tilde{p}$ is a
$D_\infty$-$\epsilon$-fibration for all $0 < \epsilon< \frac{1}{2}$.
Therefore $\tilde{p}$ is a $D_\infty$-MAF.\qed
\end{proof}

\begin{example}
Finally, we illustrate why the proper $C_2$-homotopy in Part (2)
cannot be improved to a bounded $C_2$-homotopy from $p$ to
$\tilde{p}$. Let $m \neq 0,1$ be a real number. Define a
homeomorphism
\[
p: \br \longrightarrow \br; \quad x \longmapsto mx.
\]
Clearly $p$ is a $C_2$-MAF. Assume that $p$ is boundedly
$C_2$-homotopic to a $D_\infty$-MAF $\tilde{p}: \br \to \br$. In
particular, $p$ is $\epsilon$-close to a $C_\infty$-map $\tilde{p}$
for some $\epsilon>0$. Then an elementary argument shows that
$\tilde{p}$ is $\mu$-close to the $C_\infty$-map
\[
q: \br \longrightarrow \br; \quad x \longmapsto \tilde{p}(0) + x
\]
for any $\mu > \sup_{x \in [0,1]} |\tilde{p}(x) - q(x)|$. Hence $p$
is $(\epsilon+\mu)$-close to $q$, a contradiction.
\end{example}


\begin{acknowledgements}
The authors wish to thank Jim Davis and Shmuel Weinberger, for
drawing their attention to problems related to $C_2$-manifold approximate fibrations over
the circle. Furthermore, Andr\'e Henriques and Ian Leary are appreciated for Proposition~\ref{prop:HenriquesLeary}. The authors were supported in part by NSF Grants DMS--0504176 and DMS--0904276, respectively.
\end{acknowledgements}

\bibliographystyle{abbrv}
\bibliography{DihedralMAFsOverCircle}

\begin{thebibliography}{10}

\bibitem{Beshears}
A.~Beshears.
\newblock {\em {$G$}-isovariant structure sets and stratified structure sets}.
\newblock Dissertation. Vanderbilt University, 1997.

\bibitem{BridsonHaefliger}
M.~R. Bridson and A.~Haefliger.
\newblock {\em Metric spaces of non-positive curvature}, volume 319 of {\em
  Grundlehren der Mathematischen Wissenschaften [Fundamental Principles of
  Mathematical Sciences]}.
\newblock Springer-Verlag, Berlin, 1999.

\bibitem{ChapmanTAMS}
T.~A. Chapman.
\newblock Approximation results in {H}ilbert cube manifolds.
\newblock {\em Trans. Amer. Math. Soc.}, 262(2):303--334, 1980.

\bibitem{Chapman}
T.~A. Chapman.
\newblock Approximation results in topological manifolds.
\newblock {\em Mem. Amer. Math. Soc.}, 34(251):iii+64, 1981.

\bibitem{ChapmanFerry}
T.~A. Chapman and S.~Ferry.
\newblock Approximating homotopy equivalences by homeomorphisms.
\newblock {\em Amer. J. Math.}, 101(3):583--607, 1979.

\bibitem{CoramDuvall2}
D.~Coram and P.~Duvall.
\newblock Approximate fibrations and a movability condition for maps.
\newblock {\em Pacific J. Math.}, 72(1):41--56, 1977.

\bibitem{Coram}
D.~S. Coram.
\newblock Approximate fibrations---a geometric perspective.
\newblock In {\em Shape theory and geometric topology ({D}ubrovnik, 1981)},
  volume 870 of {\em Lecture Notes in Math.}, pages 37--47. Springer, Berlin,
  1981.

\bibitem{CoramDuvall}
D.~S. Coram and P.~F. Duvall, Jr.
\newblock Approximate fibrations.
\newblock {\em Rocky Mountain J. Math.}, 7(2):275--288, 1977.

\bibitem{EdwardsKirby}
R.~D. Edwards and R.~C. Kirby.
\newblock Deformations of spaces of imbeddings.
\newblock {\em Ann. Math. (2)}, 93:63--88, 1971.

\bibitem{Farkas}
D.~R. Farkas.
\newblock Crystallographic groups and their mathematics.
\newblock {\em Rocky Mountain J. Math.}, 11(4):511--551, 1981.

\bibitem{Farrell_ICM}
F.~T. Farrell.
\newblock The obstruction to fibering a manifold over a circle.
\newblock In {\em Actes du {C}ongr\`es {I}nternational des {M}ath\'ematiciens
  ({N}ice, 1970), {T}ome 2}, pages 69--72. Gauthier-Villars, Paris, 1971.

\bibitem{Farrell_thesis}
F.~T. Farrell.
\newblock The obstruction to fibering a manifold over a circle.
\newblock {\em Indiana Univ. Math. J.}, 21:315--346, 1971/1972.

\bibitem{Ferry_notes}
S.~Ferry.
\newblock Approximate fibrations over {$S^1$}.
\newblock {H}and-written manuscript.

\bibitem{Hughes2002}
B.~Hughes.
\newblock Products and adjunctions of manifold stratified spaces.
\newblock {\em Topology Appl.}, 124(1):47--67, 2002.

\bibitem{HughesPrassidis}
B.~Hughes and S.~Prassidis.
\newblock Control and relaxation over the circle.
\newblock {\em Mem. Amer. Math. Soc.}, 145(691):x+96, 2000.

\bibitem{HughesRanicki}
B.~Hughes and A.~Ranicki.
\newblock {\em Ends of complexes}, volume 123 of {\em Cambridge Tracts in
  Mathematics}.
\newblock Cambridge University Press, Cambridge, 1996.

\bibitem{Hughes1981}
C.~B. Hughes.
\newblock {\em Local homotopy properties in spaces of approximate fibrations}.
\newblock Dissertation. University of Kentucky, 1981.

\bibitem{Hughes1985T}
C.~B. Hughes.
\newblock Approximate fibrations on topological manifolds.
\newblock {\em Michigan Math. J.}, 32(2):167--183, 1985.

\bibitem{Hughes1985Q}
C.~B. Hughes.
\newblock Spaces of approximate fibrations on {H}ilbert cube manifolds.
\newblock {\em Compositio Math.}, 56(2):131--151, 1985.

\bibitem{HTW1990}
C.~B. Hughes, L.~R. Taylor, and E.~B. Williams.
\newblock Bundle theories for topological manifolds.
\newblock {\em Trans. Amer. Math. Soc.}, 319(1):1--65, 1990.

\bibitem{HTW1993}
C.~B. Hughes, L.~R. Taylor, and E.~B. Williams.
\newblock Bounded homeomorphisms over {H}adamard manifolds.
\newblock {\em Math. Scand.}, 73(2):161--176, 1993.

\bibitem{Illman_Finite}
S.~Illman.
\newblock Smooth equivariant triangulations of {$G$}-manifolds for {$G$} a
  finite group.
\newblock {\em Math. Ann.}, 233(3):199--220, 1978.

\bibitem{Jaworowski}
J.~W. Jaworowski.
\newblock Extensions of {$G$}-maps and {E}uclidean {$G$}-retracts.
\newblock {\em Math. Z.}, 146(2):143--148, 1976.

\bibitem{KirbySiebenmann}
R.~C. Kirby and L.~C. Siebenmann.
\newblock {\em Foundational essays on topological manifolds, smoothings, and
  triangulations}.
\newblock Princeton University Press, Princeton, N.J., 1977.
\newblock With notes by John Milnor and Michael Atiyah, Annals of Mathematics
  Studies, No. 88.

\bibitem{MardesicSegal}
S.~Marde{\v{s}}i{\'c} and J.~Segal.
\newblock {\em Shape theory}, volume~26 of {\em North-Holland Mathematical
  Library}.
\newblock North-Holland Publishing Co., Amsterdam, 1982.
\newblock The inverse system approach.

\bibitem{Prassidis}
S.~Prassidis.
\newblock Equivariant approximate fibrations.
\newblock {\em Forum Math.}, 7(6):755--779, 1995.

\bibitem{RatcliffeBook}
J.~G. Ratcliffe.
\newblock {\em Foundations of hyperbolic manifolds}, volume 149 of {\em
  Graduate Texts in Mathematics}.
\newblock Springer, New York, second edition, 2006.

\bibitem{Siebenmann}
L.~C. Siebenmann.
\newblock A total {W}hitehead torsion obstruction to fibering over the circle.
\newblock {\em Comment. Math. Helv.}, 45:1--48, 1970.

\bibitem{tomDieck}
T.~tom Dieck.
\newblock {\em Transformation groups}, volume~8 of {\em de Gruyter Studies in
  Mathematics}.
\newblock Walter de Gruyter \& Co., Berlin, 1987.

\end{thebibliography}

\end{document}